\documentclass[11pt]{article}
\usepackage{fullpage,amsthm,enumitem}
\usepackage{amsmath,amsfonts,amssymb,graphicx} 

\usepackage{mathpazo} 
\usepackage[dvipsnames]{xcolor} 
\usepackage[hyperfootnotes=false]{hyperref}
\usepackage[title]{appendix}
\usepackage{tensor}
\usepackage[T1]{fontenc}
\usepackage[utf8]{inputenc}

\usepackage[backend=biber,
  style=alphabetic,
  giveninits=true,
  maxnames=10,
  maxalphanames=4,
  doi=false,
  url=false,
  isbn=false,
  eprint=false,
]{biblatex}
\renewbibmacro{in:}{}
\DeclareFieldFormat
[article,inbook,incollection,inproceedings,patent,thesis,unpublished]
  {title}{#1\isdot}

\newtheorem{thm}{Theorem}
\newtheorem{cor}[thm]{Corollary}
\newtheorem{lem}[thm]{Lemma}
\newtheorem{prop}[thm]{Proposition}
\newtheorem{defn}[thm]{Definition}
\newtheorem{ex}[thm]{Example}
\newtheorem{rem}[thm]{Remark}
\newtheorem{nota}[thm]{Notation}

\numberwithin{thm}{section}

\newcommand{\N}{\mathbb{N}}
\newcommand{\Z}{\mathbb{Z}}
\newcommand{\R}{\mathbb{R}}
\newcommand{\C}{\mathbb{C}}

\newcommand{\Hc}{\mathcal{H}}

\newcommand{\A}{\mathcal{A}}

\newcommand{\mB}{\mathcal{B}}

\DeclareMathOperator{\Tr}{Tr}
\DeclareMathOperator{\dom}{dom}

\newcommand{\op}{\operatorname{op}}

\newcommand{\sym}{\operatorname{sym}}
\newcommand{\syml}{\operatorname{sym}_{L_2}}
\def\multinom#1#2{\ensuremath{\left(\kern-.3em\left(\genfrac{}{}{0pt}{}{#1}{#2}\right)\kern-.3em\right)}}

\makeatletter
\newcommand*{\mint}[1]{%
  \mint@l{#1}{}%
}
\newcommand*{\mint@l}[2]{%
  \@ifnextchar\limits{%
    \mint@l{#1}%
  }{%
    \@ifnextchar\nolimits{%
      \mint@l{#1}%
    }{%
      \@ifnextchar\displaylimits{%
        \mint@l{#1}%
      }{%
        \mint@s{#2}{#1}%
      }%
    }%
  }%
}
\newcommand*{\mint@s}[2]{%
  \@ifnextchar_{%
    \mint@sub{#1}{#2}%
  }{%
    \@ifnextchar^{%
      \mint@sup{#1}{#2}%
    }{%
      \mint@{#1}{#2}{}{}%
    }%
  }%
}
\def\mint@sub#1#2_#3{%
  \@ifnextchar^{%
    \mint@sub@sup{#1}{#2}{#3}%
  }{%
    \mint@{#1}{#2}{#3}{}%
  }%
}
\def\mint@sup#1#2^#3{%
  \@ifnextchar_{%
    \mint@sup@sub{#1}{#2}{#3}%
  }{%
    \mint@{#1}{#2}{}{#3}%
  }%
}
\def\mint@sub@sup#1#2#3^#4{%
  \mint@{#1}{#2}{#3}{#4}%
}
\def\mint@sup@sub#1#2#3_#4{%
  \mint@{#1}{#2}{#4}{#3}%
}
\newcommand*{\mint@}[4]{%
  \mathop{}%
  \mkern-\thinmuskip
  \mathchoice{%
    \mint@@{#1}{#2}{#3}{#4}%
        \displaystyle\textstyle\scriptstyle
  }{%
    \mint@@{#1}{#2}{#3}{#4}%
        \textstyle\scriptstyle\scriptstyle
  }{%
    \mint@@{#1}{#2}{#3}{#4}%
        \scriptstyle\scriptscriptstyle\scriptscriptstyle
  }{%
    \mint@@{#1}{#2}{#3}{#4}%
        \scriptscriptstyle\scriptscriptstyle\scriptscriptstyle
  }%
  \mkern-\thinmuskip
  \int#1%
  \ifx\\#3\\\else_{#3}\fi
  \ifx\\#4\\\else^{#4}\fi  
}
\newcommand*{\mint@@}[7]{%
  \begingroup
    \sbox0{$#5\int\m@th$}%
    \sbox2{$#5\int_{}\m@th$}%
    \dimen2=\wd0 %
    \let\mint@limits=#1\relax
    \ifx\mint@limits\relax
      \sbox4{$#5\int_{\kern1sp}^{\kern1sp}\m@th$}%
      \ifdim\wd4>\wd2 %
        \let\mint@limits=\nolimits
      \else
        \let\mint@limits=\limits
      \fi
    \fi
    \ifx\mint@limits\displaylimits
      \ifx#5\displaystyle
        \let\mint@limits=\limits
      \fi
    \fi
    \ifx\mint@limits\limits
      \sbox0{$#7#3\m@th$}%
      \sbox2{$#7#4\m@th$}%
      \ifdim\wd0>\dimen2 %
        \dimen2=\wd0 %
      \fi
      \ifdim\wd2>\dimen2 %
        \dimen2=\wd2 %
      \fi
    \fi
    \rlap{%
      $#5%
        \vcenter{%
          \hbox to\dimen2{%
            \hss
            $#6{#2}\m@th$%
            \hss
          }%
        }%
      $%
    }%
  \endgroup
}

\title{
\Large 
A noncommutative integral on spectrally truncated spectral triples, and a link with quantum ergodicity}
\date{\today}

\author{Eva-Maria Hekkelman and Edward A. McDonald}

\begin{document}

\maketitle{}
\begin{abstract}
    We propose a simple approximation of the noncommutative integral in noncommutative geometry for the Connes--Van Suijlekom paradigm of spectrally truncated spectral triples. A close connection between this approximation and the field of quantum ergodicity and work by Widom in particular immediately provides a Szeg\H{o} limit formula for noncommutative geometry. We then make a connection to the density of states. Finally, we propose a definition for the ergodicity of geodesic flow for compact spectral triples. This definition is known in quantum ergodicity as uniqueness of the vacuum state for $C^*$-dynamical systems, and for spectral triples where local Weyl laws hold this implies that the Dirac operator of the spectral triple is quantum ergodic. This brings to light a close connection between quantum ergodicity and Connes' integral formula.
\end{abstract}
Noncommutative geometry (NCG)~\cite{Connes1994} aims to study geometry through spectral data, motivated in part by the result that a Riemannian manifold can be reconstructed by such means~\cite{Connes2013}. The relevant spectral data can be studied in the form of a \textit{spectral triple}. For applications of NCG in physics and numerical computations in NCG, it is important to know how well spectral triples can be approximated by a finite truncation, since this is all we can measure physically or compute numerically. Connes and Van Suijlekom introduced the concept of operator system spectral triples for this purpose~\cite{ConnesvSuijlekom2021}, developments towards which were made in~\cite{DAndreaLizzi2014, GlaserStern2020, GlaserStern2021, ConnesvSuijlekom2022, DAndreaLandi2022, Hekkelman2022, GielenvSuijlekom2023, Rieffel2023, LeimbachvSuijlekom2024, Suijlekom2024a, Suijlekom2024b} amongst others.

We will connect this paradigm with Connes' noncommutative integral. On a Hilbert space $\Hc$ with Dirac operator $D$, 
the (normalised) positive functional
\begin{equation}\label{eq:NoncomInt}
a \mapsto \frac{\Tr_\omega(a\langle D\rangle ^{-d})}{\Tr_\omega(\langle D\rangle^{-d})}, \quad a \in B(\Hc),
\end{equation}
where $\langle x \rangle := (1+|x|^2)^{\frac{1}{2}}$, $\omega \in \ell_\infty^*$ is an extended limit, and $\Tr_\omega$ is the corresponding Dixmier trace (see Section~\ref{S:Preliminaries}), has been identified by Connes as the correct analogue in NCG of integration on compact Riemannian manifolds~\cite{Connes1994} and therefore has been dubbed the noncommutative integral. In this note we will show that given a finite-rank spectral projection $P_\lambda : = \chi_{[-\lambda,\lambda]}(D)$ where $\chi_{[-\lambda,\lambda]}$ is the indicator function of the interval $[-\lambda, \lambda]\subseteq \R$, the functional
\begin{equation}\label{eq:TruncInt}
    P_\lambda a P_\lambda \mapsto \frac{\Tr(P_\lambda a P_\lambda)}{\Tr(P_\lambda)}, \quad a \in B(\Hc),
\end{equation}
approximates the noncommutative integral~\eqref{eq:NoncomInt} on spectrally truncated unital spectral triples (Proposition~\ref{P:HeattoInt}, Theorem~\ref{T:Log-CesaroMean}). This is a result in the spirit of~\cite{Stern2019}, where finite-rank approximations of zeta residues are given. We however do not assume the existence of a full asymptotic expansion of the heat trace. Instead, we focus our efforts on the computation of the first term of this expansion, which is the noncommutative integral.

The language involved is closely tied to the field of quantum ergodicity, the inception of which can largely be credited to Shnirelman, Zelditch and Colin de Verdi\`ere~\cite{Shnirelman1974,ColindeVerdiere1985,Zelditch1987}. For reviews of this field, we refer to~\cite{Zelditch2010, Zelditch2017}. Quantum ergodicity is a property of an operator which can mean various things. A common definition is that, given a compact Riemannian manifold $M$ and a positive self-adjoint operator $\Delta$ on $L_2(M)$ with compact resolvent, the operator $\Delta$ is said to be quantum ergodic if for every orthonormal basis $\{e_n\}_{n=0}^\infty$ of $L_2(M)$ consisting of eigenfunctions of $\Delta$ with non-decreasing corresponding eigenvalues, there exists a density one subsequence $J\subseteq \N$ such that for all zero-order classical pseudodifferential operators $\mathrm{Op}(\sigma)$ with principal symbol $\sigma \in C^\infty(S^*M)$,
\[
\lim_{J \ni j \to \infty}\langle e_{j}, \mathrm{Op}(\sigma) e_{j}\rangle_{L_2(M)} = \int_{S^*M} \sigma \, d\nu,
\]
where $\nu$ is the measure on the cotangent sphere $S^*M$ induced by the Riemannian metric. In this context, a density one subsequence means that 
\[
\frac{\# J \cap \{0,\ldots, n\}}{n+1} \to 1, \quad n \to \infty.
\]
Quantum ergodicity implies in particular that the eigenfunctions $|e_j|^2$ become uniformly distributed over $M$ as $J\ni j \to \infty$, in the sense that the measures $|e_j|^2 d\nu_g$ converge to $\frac{1}{\mathrm{vol}(M)}d\nu_g$ in the weak$^*$-topology.

Although quantum ergodicity shares a philosophical link with NCG -- emerging from a functional-analytic approach to ergodic geodesic flow on compact Riemannian manifolds -- there has yet to be made an explicit connection between the two fields, despite their contemporary development. We will show in Section~\ref{S:Ergodicity} that our results on the noncommutative integral on truncated spectral triples provide the means with which the gap can be bridged.

We propose below a straightforward noncommutative generalisation of the property of ergodic geodesic flow on compact Riemannian manifolds for spectral triples, and explore what some results from the field of quantum ergodicity provide in this context. Our definition of ergodicity is known in the study of $C^*$-dynamical systems as uniqueness of the vacuum state, and hence a result by Zelditch~\cite{Zelditch1996} can now be recognised as an NCG version of the classical result that ergodicity of the geodesic flow implies quantum ergodicity of the Laplace--Beltrami operator~\cite{Shnirelman1974,ColindeVerdiere1985,Zelditch1987}, see Theorem~\ref{T:MainQE} below.

Additionally, we will draw from a result of Widom~\cite{Widom1979} on the asymptotic behaviour of the functional~\eqref{eq:TruncInt}, which directly implies a Szeg\H{o} limit formula for spectral triples that satisfy the Weyl law (Theorem~\ref{T:Szego}). This provides that for all self-adjoint $A \in B(\Hc)$ which map $\dom|D|$ into itself and such that $[D,A]$ is bounded,
\[
\Tr_\omega (\langle D \rangle^{-d}) \cdot \omega \circ M \bigg(  \frac{\Tr(f(P_{\lambda_n} A P_{\lambda_n}))}{\Tr(P_{\lambda_n})}\bigg) = \Tr_\omega (f(A)\langle D\rangle^{-d}), \quad f \in C(\R),\,f(0)=0.
\]
Here, $M: \ell_\infty \to \ell_\infty$ is a logarithmic averaging operator, and $\omega \in \ell_\infty^*$ is an extended limit. 
Details are provided in Section~\ref{S:Szego}. Note that we use the short-hand notation $\omega \circ M (a_n)$ for $ \omega \circ M (\{a_n\}_{n=1}^\infty)$. We remark that this result provides the insight that Szeg\H{o} limit theorems can be interpreted as versions of Connes' integral formula.

An outline of this paper is as follows. We start with some preliminaries in Section~\ref{S:Preliminaries}, and we then explore and make precise the relation between the functionals~\eqref{eq:NoncomInt} and~\eqref{eq:TruncInt} in Section~\ref{S:Integration}.
Section~\ref{S:Szego} provides the mentioned Szeg\H{o} limit theorem for NCG. Next, we discuss a way of interpreting the functional~\eqref{eq:TruncInt} when the noncommutative integral~\eqref{eq:NoncomInt} is not defined, for example in $\theta$-summable or $\mathrm{Li}_1$-summable spectral triples. Namely, we relate the functional~\eqref{eq:TruncInt} to a functional that is sometimes called the Fr\"ohlich functional, which has been studied extensively in~\cite{GoffengRennie2019} as a KMS state. Finally, in Section~\ref{S:Ergodicity} we exhibit our study in quantum ergodicity and its relation to NCG through our results on the noncommutative integral.

\textbf{Acknowledgement:} The authors thank Nigel Higson for helpful comments and suggestions, and we thank Magnus Goffeng for pointing out the condition $[D,A]$ being bounded is sufficient for Lemma~\ref{L:Widom}. We are furthermore indebted to the anonymous reviewers of this note, who provided significant insights and comments. The first named author would like to thank the Pennsylvania State University and Nigel Higson for their hospitality during a visit where part of this work was done, and is partially supported by the Australian Research Council Laureate Fellowship FL170100052. We also extend our thanks to Eric Leichtnam, Qiaochu Ma, and Rapha\"el Ponge for their assistance.

\section{Preliminaries}\label{S:Preliminaries}
A spectral triple is a construction that is modeled after the data needed to reconstruct compact Riemannian spin manifolds~\cite{Connes1994, Connes2013}. The origin of the definition can be traced to Baaj--Julg \cite{BaajJulg1983}. An operator system spectral triple is a generalisation of this, introduced in~\cite{ConnesvSuijlekom2021}.
\begin{defn}
    An operator system spectral triple $(\A, \Hc, D)$ consists of a space $\A$ of bounded operators on a Hilbert space $\Hc$ such that its norm closure is $*$-invariant, $D$ is a self-adjoint operator on $\Hc$ with compact resolvent, and for all $T \in \A$ we have that $T(\dom D) \subseteq \dom D$ and $[D,T]$ extends to a bounded operator. 
    If $\A$ forms a $*$-algebra, $(\A, \Hc, D)$ is simply called a spectral triple.
\end{defn}

Typical examples of operator system spectral triples are of the form
\[
(P\A P, P\Hc, PD),
\]
where $(\A, \Hc, D)$ is a spectral triple and $P = \chi_I(D)$ is a spectral projection of $D$. 

It should be remarked that in this work, we will mainly concentrate on
high energy (large eigenvalue) asymptotics corresponding to $D$, using spectral projections of the form
\[
P_\lambda:=\chi_{[-\lambda,\lambda]}(D) = \chi_{[0,\lambda]}(|D|).
\]
As such, our results really only depend on the positive operator $|D|$. We will keep the operator $D$ itself around to maintain notation consistent with the noncommutative geometry literature.

Write $K(\Hc)\subseteq B(\Hc)$ for the compact operators on $\Hc$, and for a compact operator $A$ write $\{\lambda(k,A) \}_{k=0}^\infty$ for any sequence of eigenvalues of $A$, counting multiplicities, ordered in decreasing modulus. For integration formulas in noncommutative geometry, an essential role is played by the weak trace-class operators (sometimes dubbed `infinitesimals of order $1$')
\[
\mathcal{L}_{1, \infty} := \{A \in K(\Hc) : \lambda(k,|A|) = O(k^{-1}),\; k\to\infty \},
\]
and the rich structure of traces on this two-sided ideal in $B(\Hc)$~\cite{LSZVol1, LMSZVol2}. 

\begin{defn}
    A singular trace on a two-sided ideal $J\subseteq B(\Hc)$ is a unitarily invariant linear functional $\phi: J \to \mathbb{C}$ that vanishes on finite-rank operators. 

    An extended limit $\omega$ is a state on $\ell_\infty$, the space of bounded sequences, which vanishes on the set of sequences converging to zero. Typically we write $\omega(a_n)$ to mean $\omega(\{a_n\}_{n=0}^\infty)$. For any extended limit $\omega \in \ell_\infty^*$, the mapping $\Tr_\omega : \mathcal{L}_{1,\infty} \to \mathbb{C}$ defined by
    \[
    \Tr_\omega(A) := \omega \bigg(\frac{1}{\log(n+2)} \sum_{k=0}^n \lambda(k,A) \bigg), \quad A \in \mathcal{L}_{1,\infty},
    \]
    is a singular trace on $\mathcal{L}_{1,\infty}$, called a Dixmier trace (see e.g.~\cite[Theorem~6.1.2]{LSZVol1}).
\end{defn}
It is important to remark that not all singular traces on $\mathcal{L}_{1,\infty}$ are Dixmier traces~\cite[Chapter~6]{LSZVol1}.
As a warning to the reader, in the literature sometimes Dixmier traces are considered on the Dixmier--Macaev ideal $\mathcal{M}_{1,\infty} := \{A \in K(\Hc) : \sum_{k=0}^n \lambda(k,|A|) = O(\log(n))\}$, which is sometimes also denoted by $\mathcal{L}_{1,\infty}$, though we have $\mathcal{L}_{1,\infty} \subsetneq \mathcal{M}_{1,\infty}.$

For spectral triples, the term `noncommutative integral' is inspired by the following result, a consequence of Connes' trace formula~\cite{Connes1988}. For details, see~\cite[Chapters~7,8]{LSZVol1} and~\cite[Chapters~2,3]{LMSZVol2}, as well as~\cite{ZaninSukochev2023}.

\begin{thm}[Connes' integration formula]\label{T:ConnesIntegration}
    Let $(M,g)$ be a $d$-dimensional closed Riemannian manifold ($d>1$) with volume form $\nu_g$. For $f \in C(M)$ and for every positive normalised trace $\phi$ on $\mathcal{L}_{1,\infty}(L_2(M))$ we have
    \[
    \phi(M_f (1-\Delta_g)^{-\frac{d}{2}}) = \phi( (1-\Delta_g)^{-\frac{d}{2}}) \int_{M} f \, d\nu_g = \frac{\operatorname{Vol}(\mathbb{S}^{d-1})}{d(2\pi)^d} \int_{M} f \, d\nu_g.
    \]
\end{thm}

A similar result holds for $\mathbb{R}^d$~\cite[Theorem~3.1.1]{LMSZVol2}. This motivates the convention that the functional
\[
a \mapsto \phi(a\langle D\rangle^{-d}), \quad a \in \A,
\]
is called the noncommutative integral for $d$-dimensional spectral triples $(\A, \Hc, D)$  (recall the notation $\langle x \rangle = (1+|x|^2)^{\frac{1}{2}}$). In NCG literature the most studied case is the one where $\phi$ is a Dixmier trace, and we too will focus on this. 
We refer to~\cite{LordPotapov2010, LordSukochev2010, LordSukochev2011, LSZVol1, LMSZVol2} for thorough studies of the noncommutative integral.

Finally, we recall two important spectral triples which we will use as examples in Section~\ref{S:Ergodicity}. For more details on the noncommutative torus (on which there is much literature) we refer to~\cite{HaLee2019,HaLee2019b} and~\cite[Section 12.3]{GVF2001}, for details on almost commutative manifolds to~\cite[Chapter~10]{Suijlekom2025}.

\begin{defn}
    Let $d\geq 2$ and let $\theta$ be a real $d \times d$ antisymmetric matrix. The noncommutative torus is the universal $C^*$-algebra $C(\mathbb{T}^d_\theta)$ generated by a family of unitary elements $\{u_n\}_{n\in \Z^d}$ subject to the relations
    \[
    u_n u_m = e^{\frac{i}{2}\langle n, \theta m\rangle}u_{n+m}, \quad n,m\in \Z^d.
    \]
    The functional
    \[
    \tau_\theta \bigg(\sum_{k \in \mathbb{Z}^d}c_k u_k \bigg):= c_0
    \]
    extends to a continuous faithful trace state on $C(\mathbb{T}^d_\theta).$ 
    The smooth subspace $C^\infty(\mathbb{T}^d_\theta)$ is the subalgebra of $x \in C(\mathbb{T}^d_\theta)$ for which $\widehat{x}(k) = \tau_\theta(xu_k^*)$ is a rapidly decaying sequence on $\Z^d.$
    The Hilbert space in the GNS representation corresponding to $\tau_\theta$ is denoted $L_2(\mathbb{T}^d_\theta),$ and $\{u_n\}_{n\in \Z^d}$ is an orthonormal basis for $L_2(\mathbb{T}^d_\theta).$ The self-adjoint densely defined operators $D_j$, $j=1, \ldots, d$ on $L_2(\mathbb{T}^d_\theta)$ are defined on the basis by
    \[
    D_j u_k := k_j u_k,\quad k=(k_1,\ldots,k_d)\in \Z^d.
    \]
    The operator $D = \sum_{j=1}^d D_j \otimes \gamma_j$ on $L_2(\mathbb{T}^d_\theta)\otimes \mathbb{C}^{N_d}$, where $\gamma_j$ are standard Clifford matrices on $\mathbb{C}^{N_d}$ with $N_d = 2^{\lfloor\frac{d}{2}\rfloor}$, gives a spectral triple
    \[
    (C^\infty(\mathbb{T}^d_\theta),L_2(\mathbb{T}^d_\theta)\otimes \mathbb{C}^{N_d}, D ),
    \]
    where we represent $C^\infty(\mathbb{T}^d_\theta)$ as operators on $L_2(\mathbb{T}^d_\theta)\otimes \mathbb{C}^{N_d}$ by acting on the first component~\cite[Section 12.3]{GVF2001}.
    We write $\Delta := -\sum_{j=1}^d D_j^2$ as an operator on $L_2(\mathbb{T}^d_\theta)$, so that $|D|= \sqrt{-\Delta} \otimes 1_{\mathbb{C}^{N_d}}$. 
\end{defn}

\begin{defn}
    A spectral triple $(\A, \Hc, D)$ is called even if equipped with a $\Z_2$-grading $\gamma$ on $\Hc$ such that $D\gamma = -\gamma D$ and $a\gamma = \gamma a$ for all $a \in \A$. The canonical spectral triple $(C^\infty(M), L_2(S), D_M)$ of an even-dimensional Riemannian spin manifold has a natural grading $\gamma_M$ making it an even spectral triple. Given such an even-dimensional manifold and a finite spectral triple $(\A_F, \Hc_F, D_F)$, meaning that $\Hc_F$ and $\A_F$ are finite-dimensional, we define the product spectral triple
    \[
    (C^\infty(M) \otimes \A_F, L_2(S) \otimes \Hc_F, D_M \otimes 1 + \gamma_M \otimes D_F).
    \]
    This spectral triple is called an almost-commutative manifold. 
\end{defn}

\section{Integration on truncated spectral triples}
\label{S:Integration}
Let us fix a closed self-adjoint operator $D$ on a separable Hilbert space $\Hc$ such that $\langle D\rangle^{-d}\in \mathcal{L}_{1,\infty}$, where $d>0$ and $\langle x \rangle := (1+|x|^2)^{\frac{1}{2}}$. We fix an extended limit $\omega \in \ell_{\infty}^*$ and assume that $\Tr_\omega(\langle D\rangle^{-d}) > 0$.
We write $P_\lambda := \chi_{[-\lambda, \lambda]}(D)$. This situation is modeled after (compact) $d$-dimensional spectral triples $(\A, \Hc, D)$.

We first provide the most straight-forward approach to the noncommutative integral on truncated triples, using standard techniques that are employed in quantum ergodicity~\cite{ColindeVerdiere1985}. We write 
\[
f(t) \sim C t^{-\alpha}
\]
to mean
\[
\lim_{t\to 0} t^{\alpha}f(t) = C.
\]

\begin{prop}\label{P:HeattoInt}
    Let $a \in B(\Hc)$. If there exist constants $C, C(a) \in \R$ with 
    \[
    \Tr(e^{-tD^2}) \sim C t^{-\frac{d}{2}} , \quad \Tr(a e^{-tD^2}) \sim C(a)t^{-\frac{d}{2}},
    \]
    then
    \[
    \frac{\Tr_\omega (a \langle D\rangle^{-d})}{\Tr_\omega(\langle D \rangle^{-d})} =  \lim_{\lambda \to \infty}\frac{\Tr(P_\lambda a P_\lambda)}{\Tr (P_\lambda)}.
    \]
\end{prop}
\begin{proof}
    By~\cite[Corollary~8.1.3]{LSZVol1} we have that 
    \[
    C = \Gamma(\frac{d}{2}+1) \Tr_\omega(\langle D\rangle^{-d}), \quad C(a) = \Gamma(\frac{d}{2}+1) \Tr_\omega(a \langle D\rangle^{-d}).
    \]
    Recall that we assume $\Tr_\omega(\langle D\rangle^{-d})>0$. An application of the Hardy--Littlewood Tauberian theorem~\cite[Theorem~XII.5.2]{Feller1971} to the function $\Tr(e^{-tD^2})$ shows that 
    \[
    \Tr(P_\lambda) \sim \Tr_\omega(\langle D\rangle^{-d}) \lambda^{d}, \quad \lambda \to \infty.
    \] Applying the theorem again to the function $\Tr(a e^{-tD^2})$ then gives that $\lim_{\lambda \to \infty}\frac{\Tr(P_\lambda a P_\lambda)}{\Tr (P_\lambda)}$ exists and is equal to $\frac{\Tr_\omega(a\langle D\rangle^{-d})}{\Tr_\omega(\langle D\rangle^{-d})}$. 
\end{proof}

\begin{rem}   
    The Hardy-Littlewood Tauberian theorem implies that the condition $\Tr(e^{-tD^2})\sim Ct^{-\frac{d}{2}}$ as $t\to 0$ is equivalent to $\lambda(k,D^2)\sim \widetilde{C} k^{\frac{2}{d}}$ as $k\to\infty$~\cite[Theorem~XII.5.2]{Feller1971}.
    \end{rem}

\begin{defn}\label{D:WeylLaw}
    We say that $D^2$ (as fixed at the start of this section) satisfies a Weyl law if $\Tr(e^{-tD^2})\sim Ct^{-\frac{d}{2}}$, and it satisfies a local Weyl law for an operator $a \in B(\Hc)$ if $ \Tr(a e^{-tD^2}) \sim C(a)t^{-\frac{d}{2}}$. 
\end{defn}

    See~\cite{McDonaldSukochev2022WeylLaw} for an investigation of the validity of the (local) Weyl law for spectral triples, and~\cite{Ponge2023} for an extensive study of Weyl's law in relation to Connes' integral formula. The latter, work by Ponge, answers some questions regarding Weyl laws and the noncommutative integral related to measurability of operators.

Although the local Weyl laws hold for Riemannian manifolds and a wide class of spectral triples~\cite{GrubbSeeley1995,Vassilevich2007, EcksteinZajac2015, McDonaldSukochev2022WeylLaw}, there are spectral triples in which such behaviour does not hold, see for example~\cite[Example~5.7]{MOOIs}. In the remainder of this section we show what can be deduced without this condition.
We now fix an orthonormal basis $\{e_n\}_{n=0}^\infty$ of eigenvectors of $|D|$, ordered such that the corresponding eigenvalues $\{\lambda_n\}_{n=0}^\infty$ are non-decreasing. The following lemma is closely related to~\cite[Theorem~3.6]{LordSukochev2011}.

\begin{lem}\label{L:WeakestInt}
    Let $A \in B(\Hc)$. Then
    \[
    \Tr_\omega(A\langle D\rangle^{-d}) = \omega\bigg(\frac{1}{\log(n+2)} \sum_{k=0}^n \langle\lambda_k\rangle^{-d} \langle e_k, A e_k \rangle\bigg).
    \]
    If $D^2$ satisfies Weyl's law, i.e. $\lambda_k \sim C k^{\frac{1}{d}}$, this simplifies to
    \[
    \frac{\Tr_\omega (A\langle D\rangle^{-d})}{\Tr_\omega(\langle D\rangle^{-d})} =  \omega\bigg(\frac{1}{\log(n+2)} \sum_{k=0}^n  \frac{\langle e_k, A e_k \rangle}{k+1} \bigg).
    \]
\end{lem}
\begin{proof}
    The first part is~\cite[Corollary~7.1.4(c)]{LSZVol1}, the second claim is~\cite[Theorem~7.1.5(a)]{LSZVol1}.
\end{proof}

What appears in the lemma above is the logarithmic mean $M: \ell_\infty \to \ell_\infty$, defined by
\begin{align*}
    M: x &\mapsto \bigg\{\frac{1}{\log(n+2)} \sum_{k=0}^n \frac{x_k}{k+1} \bigg\}_{n=0}^\infty.
\end{align*}
This can be compared with the Ces\`aro mean
\begin{align*}
    C: x &\mapsto \bigg\{\frac{1}{n+1} \sum_{k=0}^n x_k \bigg\}_{n=0}^\infty.
\end{align*}

\begin{lem}\label{L:LogToCes}
    For any sequence $x \in \ell_\infty$, we have
    \[
    (M(x))_n = (M \circ C(x))_n + o(1), \quad n \to \infty.
    \]
\end{lem}
\begin{proof}
    For $x \in \ell_\infty$ and $k \geq 0$ we have
    \begin{align*}
         \frac{x_k}{k+1}  &= \bigg(\frac{1}{k+1}\sum_{l=0}^k x_l\bigg) - \frac{k}{k+1}\bigg(\frac{1}{k}\sum_{l=0}^{k-1} x_l\bigg)\\
         &= (C(x))_k - (C(x))_{k-1} + \frac{1}{k+1} (C(x))_{k-1}.
    \end{align*}
    Hence, as $n\to \infty$
    \begin{align*}
        (M(x))_n &= \frac{1}{\log(n+2)} \sum_{k=0}^n \frac{x_k}{k+1}\\
        &= \frac{1}{\log(n+2)}\bigg( (C(x))_n + \sum_{k=1}^n \frac{1}{k+1} (C(x))_{k-1} \bigg)\\
        &= (M\circ T \circ C(x))_n + o(1),
    \end{align*}
    where $T : (x_0, x_1, x_2, \ldots) \mapsto (0, x_0, x_1, \ldots)$ is the right-shift operator on $\ell_\infty$. Finally, for any bounded sequence $a \in \ell_\infty$, we have that
    \[
    (M\circ T(a))_n - (M(a))_n = o(1), \quad n \to \infty,
    \]
    which can be found in~\cite[Lemma~6.2.12]{LSZVol1}.
\end{proof}

Since both $M$ and $C$ are regular transformations in Hardy's terminology~\cite[Chapter~III]{Hardy1949}, meaning that $M(x)_n \to c $ whenever $x_n \to c$, it is a consequence of Lemma~\ref{L:LogToCes} that for $x \in \ell_\infty$, if $C(x)_n \to c$ then $M(x)_n \to c$ as $n \to \infty$. We introduce one more crucial lemma. Namely, writing $Q_n$ for the projection onto $\{ e_0, \ldots, e_n\}$, we want to switch freely between 
\[
\frac{\Tr(P_\lambda a P_\lambda)}{\Tr(P_\lambda)}, \quad \frac{\Tr(Q_n a Q_n )}{\Tr(Q_n)}.
\]
The first can be written as $\frac{\Tr(Q_{N(\lambda)} a Q_{N(\lambda)} )}{\Tr(Q_{N(\lambda)})},$ where $N(\lambda)$ is the greatest $k\geq 0$ such that $\lambda_k\leq \lambda$, and thus can be interpreted as a subsequence of the second. The following lemma can therefore be applied, which appeared as~\cite[Lemma~4.8]{AzamovHekkelman2022} in a slightly weaker form and in a different context.

\begin{lem}\label{L:AHMSZ}
    Let $\phi: \mathbb{N} \to \R_{> 0}$ be an increasing function such that $\phi(n)\to \infty$ as $n\to \infty$, let $\{a_k\}_{k \in \mathbb{N}} \subseteq \mathbb{R}$ be a sequence such that $\big\{\frac{1}{\phi(n)}\sum_{k=0}^n |a_k| \big\}_{n=0}^\infty$ is bounded, and let $\{k_0, k_1, \dots \}$ be an infinite, increasing sequence of positive integers such that 
    \[
    \lim_{n\to \infty} \frac{\phi(k_{n+1})}{\phi(k_n)} = 1,
    \] 
    and
    \[
    \frac{1}{\phi(k_{n})} \sum_{k=k_{n-1}+1}^{k_{n}} |a_k| = o(1), \quad n\to\infty.
    \]
    Labeling $k_{i_n} := \min\{k_i : k_{i} \geq n \}$,     
    we have that
    \[
    \frac{1}{\phi(n)}\sum_{k=0}^n a_k = \frac{1}{\phi(k_{i_n})}\sum_{k=0}^{k_{i_n}} a_k + o(1), \quad n \to \infty.
    \]
\end{lem}
\begin{proof}
Without loss of generality, we can assume that $\{a_k\}_{k\in \N}$ is a positive sequence. We have \begin{align*}
    \frac{1}{\phi(n)} \sum_{k=1}^n a_k -  \frac{1}{\phi(k_{i_n})} \sum_{k=1}^{k_{i_n}} a_k 
    &\leq  \bigg( \frac{\phi(k_{i_n})}{\phi(k_{i_n-1})}-1\bigg)\frac{1}{\phi(k_{i_n})}\sum_{k=1}^{k_{i_n}} a_k =o(1);\\
    \frac{1}{\phi(k_{i_n})} \sum_{k=1}^{k_{i_n}} a_k-\frac{1}{\phi(n)} \sum_{k=1}^n a_k & \leq \frac{1}{\phi(k_{i_n})} \sum_{k=k_{i_n-1}+1}^{k_{i_n}} a_k =o(1). 
\end{align*}
\end{proof}

We can now prove the main result of this section.
\begin{thm}\label{T:Log-CesaroMean}
    Let $A \in B(\Hc)$. If $D^2$ satisfies Weyl's law (Definition~\ref{D:WeylLaw}), then
    \begin{equation}\label{eq:mainthmeq}
    \frac{\Tr_\omega(A\langle D\rangle^{-d})}{\Tr_\omega(\langle D\rangle^{-d})} =(\omega \circ M)\big(\langle e_n, A e_n \rangle\big) =(\omega \circ M) \bigg( \frac{\Tr(Q_n A Q_n)}{\Tr(Q_n)} \bigg)=
     (\omega \circ M) \bigg( \frac{\Tr(P_{\lambda_n} A P_{\lambda_n})}{\Tr(P_{\lambda_n})} \bigg).
    \end{equation}
    If furthermore $Q$ is an operator with $\bigcap_{n\geq 0}\dom(D^n) \subseteq \dom Q$ such that for some $s\geq -d,$ $Q\langle D\rangle^{-s}$ extends to a bounded operator, we have
    \begin{align}
        \Tr_\omega (Q) &= \omega \bigg(\frac{\Tr(P_{\lambda_n}QP_{\lambda_n})}{\log(\Tr(P_{\lambda_n}))} \bigg), \quad s = -d;\label{eq:mainthmeq2}\\
        \frac{\Tr_\omega (Q\langle D\rangle^{-s-d})}{\big(\Tr_\omega(\langle D\rangle^{-d})\big)^{\frac{s}{d}+1}} &= \Big(\frac{s}{d}+1\Big) \omega \circ M\bigg(\frac{\Tr(P_{\lambda_n} QP_{\lambda_n})}{\Tr(P_{\lambda_n})^{\frac{s}{d}+1}}\bigg), \quad s > -d.\label{eq:mainthmeq3}
    \end{align}
\end{thm}
\begin{proof}
    The first equality in equation~\eqref{eq:mainthmeq} appeared in Lemma~\ref{L:WeakestInt}, the second equality is a consequence of Lemma~\ref{L:LogToCes} and the trivial identity
    \[
     \frac{\Tr(Q_n A Q_n)}{\Tr(Q_n)} = \frac{1}{n+1} \sum_{k=0}^n \langle e_k, A e_k\rangle.
    \] 
    The last equality follows from Lemma~\ref{L:AHMSZ} when taking $\phi(n) = n+1$, since the Weyl law gives that $\frac{N(\lambda_n)}{N(\lambda_{n+1})} \to 1.$ The assumption
    \[
    \frac{1}{N(\lambda_n)} \sum_{k=N(\lambda_{n-1})+1}^{N(\lambda_n)} \langle e_k, A e_k\rangle = o(1), \quad n\to\infty
    \]
    in Lemma~\ref{L:AHMSZ} is satisfied, since
    \[
    \frac{1}{N(\lambda_n)} \sum_{k=N(\lambda_{n-1})+1}^{N(\lambda_n)}| \langle e_k, A e_k\rangle | \leq \|A\|_\infty\frac{N(\lambda_n) - N(\lambda_{n-1})}{N(\lambda_{n})} = o(1), \quad n \to \infty.
    \]

    Now take an operator $Q$ with $\bigcap_{n\geq 0}\dom(D^n) \subseteq \dom Q$ such that $Q\langle D\rangle^{-s}$ extends to a bounded operator. For $s=-d$, the given formula for $\Tr_\omega(Q)$, equation~\eqref{eq:mainthmeq2}, is a combination of Lemma~\ref{L:WeakestInt} and Lemma~\ref{L:AHMSZ}. For equation~\eqref{eq:mainthmeq3} we take $s \neq -d$. First, due to Weyl's law
    \begin{align*}
       \langle\lambda_k\rangle^{-s-d} = \big(\Tr_\omega(\langle D\rangle^{-d})\big)^{\frac{s}{d}+1} (k+1)^{-\frac{s}{d}-1} + o\big((k+1)^{-\frac{s}{d}-1}\big), \quad k \to \infty
    \end{align*}
    and hence, since $(k+1)^{-\frac{s}{d}}\langle e_k, Qe_k \rangle$ is bounded, we have that
    \begin{align*}
       \langle\lambda_k\rangle^{-s-d} \langle e_k, Q e_k \rangle = \big(\Tr_\omega(\langle D\rangle^{-d})\big)^{\frac{s}{d}+1} (k+1)^{-\frac{s}{d}-1} \langle e_k, Q e_k\rangle + o\big((k+1)^{-1}\big), \quad k \to \infty.
    \end{align*}
    Now applying Lemma~\ref{L:WeakestInt} and then Lemma~\ref{L:LogToCes},
    \begin{equation}\label{eq:proofeq1}\begin{split}
        \Tr_\omega (Q\langle D\rangle^{-s-d}) &=
        \omega \bigg(\frac{1}{\log(n+2)} \sum_{k \leq n} \langle\lambda_k\rangle^{-s-d}\langle e_k, Q e_k \rangle \bigg)\\
        &=\big(\Tr_\omega(\langle D\rangle^{-d})\big)^{\frac{s}{d}+1} \omega \circ M \big( (n+1)^{-\frac{s}{d}}\langle e_n, Q e_n \rangle \big)\\
        &=\big(\Tr_\omega(\langle D\rangle^{-d})\big)^{\frac{s}{d}+1} \omega \circ M \bigg( \frac{1}{n+1} \sum_{k \leq n}(k+1)^{-\frac{s}{d}}\langle e_k, Q e_k \rangle \bigg).
    \end{split}\end{equation}
    Using Abel's summation formula, as $n \to \infty$
    \begin{align*}
         \frac{1}{n+1} \sum_{k \leq n}(k+1)^{-\frac{s}{d}}\langle e_k, Q e_k \rangle  &= (n+1)^{-\frac{s}{d}-1}\sum_{k \leq n} \langle e_k, Q e_k \rangle\\
         &\quad- \frac{1}{n+1}\sum_{k \leq n-1}\big((k+2)^{-\frac{s}{d}}-(k+1)^{-\frac{s}{d}}\big) \sum_{j \leq k}\langle e_j, Q e_j \rangle.
    \end{align*}
    By Taylor's formula, we have
    \[
        (k+2)^{-\frac{s}{d}}-(k+1)^{-\frac{s}{d}}+\frac{s}{d}(k+1)^{-\frac{s}{d}-1} = \frac{s}{d}\big(\frac{s}{d}+1\big)\int_0^1 (1-\theta)(k+1+\theta)^{-\frac{s}{d}-2}\,d\theta.
    \]
    Therefore    
    \begin{align*}
         \frac{1}{n+1} \sum_{k \leq n}(k+1)^{-\frac{s}{d}}\langle e_k, Q e_k \rangle &=  (n+1)^{-\frac{s}{d}-1}\sum_{k \leq n} \langle e_k, Q e_k \rangle + \frac{s}{d} C\big( \big\{(k+1)^{-\frac{s}{d}-1} \sum_{j \leq k}\langle e_j, Q e_j \rangle\big\}_{k=0}^\infty\big)_n\\
         &\quad -\frac{s}{d}\big(\frac{s}{d}+1\big)\int_0^1 (1-\theta) C\big(\big\{(k+1+\theta)^{-\frac{s}{d}-2}\sum_{j\leq k}\langle e_j,Qe_j\rangle\big\}_{k=0}^\infty\big)_n\,d\theta,
    \end{align*}
    where $C: \ell_\infty \to \ell_\infty$ is the Ces\`aro operator. 
    Since $(j+1)^{-\frac{s}{d}}\langle e_j,Qe_j\rangle$ is bounded and $s>-d$ we have
    \[
        \big|(k+1+\theta)^{-\frac{s}{d}-2}\sum_{j\leq k}\langle e_j,Qe_j\rangle\big|=
            O((k+1)^{-1}),\quad k\to\infty.
    \]
    Thus
    \begin{align*}
        \frac{1}{n+1} \sum_{k \leq n}(k+1)^{-\frac{s}{d}}\langle e_k, Q e_k \rangle & =  (n+1)^{-\frac{s}{d}-1}\sum_{k \leq n} \langle e_k, Q e_k \rangle \\
        &\quad + \frac{s}{d} C\big( \big\{(k+1)^{-\frac{s}{d}-1} \sum_{j \leq k}\langle e_j, Q e_j \rangle\big\}_{k=0}^\infty\big)_n + O\Big(\frac{\log(n+2)}{n+1}\Big).
    \end{align*}
    Combining this with equation~\eqref{eq:proofeq1} and using Lemma~\ref{L:LogToCes} again, we have
    \begin{align*}
        \Tr_\omega (Q\langle D\rangle^{-s-d}) &= \big(\Tr_\omega(\langle D\rangle^{-d})\big)^{\frac{s}{d}+1} \big(1+\frac{s}{d} \big) \omega \circ M\bigg(\frac{1}{(n+1)^{\frac{s}{d}+1}}\sum_{k \leq n} \langle e_k, Q e_k \rangle \bigg).
        \end{align*}
    To apply Lemma~\ref{L:AHMSZ}, taking $\phi(n) = (n+1)^{\frac{s}{d}+1}$ and $k_n = N(\lambda_n)$, we need to check that
    \begin{align*}
        \frac{1}{N(\lambda_n)^{\frac{s}{d}+1}} \sum_{k=N(\lambda_{n-1})+1}^{N(\lambda_n)} |\langle e_k, Q e_k\rangle | &\lesssim \frac{1}{N(\lambda_n)^{\frac{s}{d}+1}}\sum_{k=N(\lambda_{n-1})+1}^{N(\lambda_n)} k^{\frac
    {s}{d}}\\
    &\lesssim \frac{N(\lambda_n)^{\frac{s}{d}+1}-N(\lambda_{n-1})^{\frac{s}{d}+1}}{N(\lambda_n)^{\frac{s}{d}+1}}\\
    &= o(1), \quad n \to \infty.
    \end{align*}
    Hence Lemma~\ref{L:AHMSZ} applies, and we conclude that
    \begin{align*}
        \Tr_\omega (Q\langle D\rangle^{-s-d})    &= \big(\frac{s}{d}+1\big)\big(\Tr_\omega(\langle D\rangle^{-d})\big)^{\frac{s}{d}+1} \omega \circ M\bigg(\frac{\Tr(P_{\lambda_n} QP_{\lambda_n})}{\Tr(P_{\lambda_n})^{\frac{s}{d}+1}}\bigg).
    \end{align*}
\end{proof}

As an obvious consequence of Theorem~\ref{T:Log-CesaroMean}, if for $A \in B(\Hc)$
\[
\frac{\Tr(P_\lambda AP_\lambda)}{\Tr(P_\lambda)}
\]
converges it follows that, provided $D^2$ satisfies Weyl's law, the limit must necessarily be the noncommutative integral of $A$. Furthermore, if the noncommutative integral is independent of $\omega$, meaning that $A\langle D\rangle^{-d}$ is Dixmier measurable (see e.g.~\cite{LSZVol1,LMSZVol2,Ponge2023}) one can replace $\omega \circ M$ by $\lim \circ M$ on the right hand sides of Theorem~\ref{T:Log-CesaroMean}. Finally, with a Weyl law, for self-adjoint $A \in B(\Hc)$ we have
\begin{align*}
\liminf_{k \to \infty} \langle e_k, A e_k \rangle \leq \liminf_{\lambda \to \infty} \frac{\Tr(P_{\lambda} A P_{\lambda})}{\Tr(P_{\lambda})} \leq \frac{\Tr_\omega(A\langle D\rangle^{-d})}{\Tr_\omega(\langle D \rangle^{-d})} \leq \limsup_{\lambda \to \infty} \frac{\Tr(P_{\lambda} A P_{\lambda})}{\Tr(P_{\lambda})} \leq \limsup_{k \to \infty} \langle e_k, A e_k \rangle.  
\end{align*}

All results achieved in this section are different flavours of the observation that the noncommutative integral is the limit point --- in a weak, averaging notion --- of the sequence $\{\langle e_k,  A e_k \rangle \}_{k=0}^\infty$. For the circle $\mathbb{T}$ this is not surprising; given $f =\sum_{k=-\infty}^\infty a_k e_k \in L_1(\mathbb{T})$ in Fourier basis, we have for \textit{every} $k \in \Z$
\[
\langle  e_k, M_f e_k \rangle = a_0 = \int_{\mathbb{T}} f(t) \, dt.
\]
More generally, Proposition~\ref{P:HeattoInt} combined with Connes' integral formula (Theorem~\ref{T:ConnesIntegration}) and Lemma~\ref{L:AHMSZ} shows that for any $d$-dimensional closed Riemannian manifold $M$ with volume form $\nu_g$ we have that the Ces\`aro mean of the sequence 
\[
\langle e_k, M_f  e_k \rangle, \quad f \in C(M)
\]
converges to $ \int_M f\, d\nu_g$. This fact is precisely what started investigations into quantum ergodicity. Recall that this covers the study of to what extent the matrix elements $\langle e_k,  M_f e_k \rangle$ themselves converge to an integral of $f$. More details will be provided in Section~\ref{S:Ergodicity}.

Previously, in~\cite[Section~7.5]{LSZVol1}\cite[Example~3.10]{LordSukochev2011} it had already been observed that for spectral triples $(\A, \Hc, D)$ where $D^2$ satisfies Weyl's law that if the noncommutative integral
    \[
     \frac{\Tr_\omega(a \langle D\rangle^{-d})}{\Tr_\omega( \langle D\rangle^{-d})}, \quad a \in \A,
    \]
    is independent of $\omega$, then
    \[
    \frac{1}{\log(n+2)} \sum_{k=0}^n \frac{\langle  e_k, a e_k\rangle}{k+1}, \quad a \in \A,
    \]
    converges as $n\to \infty$, which was interpreted as being related to quantum ergodicity.

In quantum ergodicity and related fields, there is a vast literature on the properties and asymptotics of the operators $P_\lambda a P_\lambda$. Through the results established in this section, the link with Connes' integral formula unlocks this literature for study from the perspective of noncommutative geometry. One result from this cross-pollination is a Szeg\H{o} limit theorem for truncated spectral triples. 

\section{Szeg\H{o} limit theorem}\label{S:Szego}
Szeg\H{o} proved various limit theorems concerning determinants of Toeplitz matrices, inspired by a conjecture by P\'olya and after work on these determinants by Toeplitz, Caratheodory and Fej\'er, see~\cite{Szego1915} and references therein. Much later, Widom provided a generalisation of these results with a simplified proof~\cite{Widom1979}, see also~\cite{LaptevSafarov1996} for a version for elliptic selfadjoint (pseudo)differential operators on manifolds without boundary. We now provide a translation of the results of Widom into noncommutative geometry. We thank Magnus Goffeng for pointing out that  instead of requiring that $[|D|,A]$ is bounded, it suffices to assume in the following lemma that $[D,A]$ is bounded.

\begin{lem}[\cite{Widom1979}]\label{L:Widom}
    Let $D^2$ satisfy Weyl's law (Definition~\ref{D:WeylLaw}), and let $A\in B(\Hc)$ and $B \in B(\Hc)$ map $\dom|D|$ into itself, and be such that $[D,A]$ and $[D,B]$ are bounded. 
    Then
    \[
    \lim_{\lambda \to \infty} \frac{\Tr(P_\lambda A (1-P_\lambda) B P_\lambda)}{\Tr(P_\lambda)} = 0.
    \]
\end{lem}
\begin{proof}
    First, $[D,A]$ being bounded implies that $[\langle D \rangle^{\frac{1}{2}}, A]$ is bounded due to the combination of~\cite[Theorem~6]{MOOIs} and~\cite[Proposition~5.1]{MOOIs} (alternatively, see~\cite[Lemma~10.13]{GVF2001}). Hence, replacing $D$ by $\langle D \rangle^{\frac{1}{2}}$, we can assume that $D$ is positive and that $[|D|,A]$ and $[|D|, B]$ are bounded. Then, by the Cauchy-Schwarz inequality, an equivalent formulation of the statement is that for every $B$ such that $[|D|,B]$ is bounded, we have  
    \[
    \lim_{\lambda \to\infty}\frac{\|P_\lambda B (1-P_\lambda)\|_{HS}^2}{\Tr(P_\lambda)} = 0,
    \]
    where $\| \cdot \|_{HS}$ is the Hilbert--Schmidt norm. The following argument is essentially due to Widom~\cite[p.~145]{Widom1979},  see also \cite[Lemma 3.4]{Guillemin1979}.
    
    If $\lambda_n, \lambda_m$ are distinct eigenvalues of $|D|$ with corresponding spectral projections $E_{\lambda_i} = \chi_{\{\lambda_i\}}(|D|)$, then
    \[
    E_{\lambda_n}BE_{\lambda_m} = \frac{1}{\lambda_n-\lambda_m}E_{\lambda_n}[|D|,B]E_{\lambda_m}. 
    \]
    For $N>0$, and writing $\{e_k\}_{k \in \N}$ for an orthonormal basis of $\Hc$ consisting of eigenvectors of $|D|$ with corresponding eigenvalues $\{\lambda_k\}_{k \in \N}$, we therefore have
    \begin{align*}
        \|P_{\lambda}B(1-P_{\lambda+N})\|_{HS}^2 &= \sum_{\substack{\lambda_n> \lambda+N\\\lambda_m\leq \lambda}} |\langle e_m, P_{\lambda}B(1-P_{\lambda+N}) e_n\rangle|^2\\
        &= \sum_{\substack{\lambda_n> \lambda+N\\\lambda_m\leq \lambda}} \frac{1}{(\lambda_n-\lambda_m)^2} |\langle e_m, P_{\lambda}[|D|,B](1-P_{\lambda+N}) e_n\rangle|^2 \\
        &\leq N^{-2}\|P_{\lambda}[|D|,B](1-P_{\lambda+N})\|_{HS}^2.
    \end{align*}
    By the triangle inequality, we have
    \begin{align*}
        \|P_{\lambda}B(1-P_{\lambda})\|_{HS}^2&\leq 2\|P_{\lambda}B(P_{\lambda+N}-P_{\lambda})\|_{HS}^2+2\|P_{\lambda}B(1-P_{\lambda+N})\|_{HS}^2\\
        &\leq 2\|B\|_{\infty}^2\Tr(P_{\lambda+N}-P_{\lambda})+2N^{-2}\Tr(P_{\lambda})\|[|D|,B]\|_{\infty}^2.
    \end{align*}
    Weyl's law implies that
    \[
        \Tr(P_{\lambda+N}-P_{\lambda}) = o(\Tr(P_{\lambda})),\quad \lambda\to\infty,
    \]
    and hence
    \[
        \limsup_{\lambda\to\infty} \frac{\|P_{\lambda}B(1-P_{\lambda})\|_{HS}^2}{\Tr(P_{\lambda})} \leq 2N^{-2}\|[|D|,B]\|_{\infty}^2.
    \]
    Since $N$ is arbitrary, this completes the proof.
\end{proof}

Following Widom~\cite{Widom1979} further, Lemma~\ref{L:Widom} can be combined with the characterisation of Connes' integral theorem in Theorem~\ref{T:Log-CesaroMean} into a Szeg\H{o} limit theorem.

\begin{thm}\label{T:Szego}
    Let $D^2$ satisfy Weyl's law (Definition~\ref{D:WeylLaw}), and let $A \in B(\Hc)$ be self-adjoint and such that it maps $\dom|D|$ into itself and $[D,A]$ is bounded. 
    Then 
    \begin{equation}\label{eq:Szegoeq1}
    (\omega \circ M) \bigg( \frac{\Tr(f(P_{\lambda_n} A P_{\lambda_n}))}{\Tr(P_{\lambda_n})}\bigg)= \frac{\Tr_\omega(f(A)\langle D\rangle^{-d})}{\Tr_\omega(\langle D \rangle^{-d})}, \quad f\in C(\R),\, f(0)=0.
    \end{equation}
    If for every positive integer $k$ there is some constant $C_k\in \R$ with
    \[
    \Tr(A^ke^{-tD^2}) \sim C_k t^{-\frac{d}{2}},
    \]
    then for every $f \in C(\R)$ with $f(0)=0$ we have
    \begin{equation}\label{eq:Szego2}
    \lim_{\lambda \to \infty}\frac{\Tr(f(P_{\lambda} A P_{\lambda}))}{\Tr(P_{\lambda})}=\frac{\Tr_\omega( f(A)\langle D\rangle^{-d})}{\Tr_\omega(\langle D\rangle^{-d})}.
    \end{equation}
\end{thm}
\begin{proof}
    To prove equation~\eqref{eq:Szegoeq1}, we sketch the proof of the stronger identity
    \begin{equation}\label{eq:StrongerSzego}
    (\omega \circ M) \bigg( \frac{\Tr( P_{\lambda_n}f(P_{\lambda_n} A P_{\lambda_n}) P_{\lambda_n})}{\Tr(P_{\lambda_n})}\bigg)= \frac{\Tr_\omega(f(A)\langle D\rangle^{-d})}{\Tr_\omega(\langle D \rangle^{-d})}, \quad f \in C(\R),
    \end{equation}
    where it is no longer necessary that $f(0)=0$.
    Lemma~\ref{L:Widom} gives that
    \begin{equation}\label{e:Widom_identity}
    \lim_{\lambda \to \infty}\frac{\Tr\big(P_\lambda A^k P_\lambda - (P_\lambda A P_\lambda)^k\big)}{\Tr(P_\lambda)} = 0, \quad k\geq 1,
    \end{equation}
    which implies equation~\eqref{eq:StrongerSzego} for polynomial $f$ through Theorem~\ref{T:Log-CesaroMean}. An application of the Stone--Weierstrass theorem provides an extension to continuous functions.  Details can be found in~\cite[p.~144]{Widom1979}.  

    Equation~\eqref{eq:Szego2} for polynomial functions $f$ is a combination of \eqref{e:Widom_identity} and Proposition~\ref{P:HeattoInt}. If $f$ is a continuous function on $\R$ with $f(0)=0$, let $\varepsilon>0$ and choose a polynomial function $p$ with $p(0)=0$ such that 
    \[
        \|f-p\|_{L_{\infty}([-\|A\|_{\infty},\|A\|_{\infty}])}<\varepsilon.
    \]
    Then 
    \[
        \Big|\frac{\Tr((f-p)(P_{\lambda}AP_{\lambda}))}{\Tr(P_{\lambda})}\Big|< \varepsilon
    \]
    and
    \[
        |\Tr_{\omega}((f-p)(A)\langle D\rangle^{-d})| \leq \varepsilon\|\langle D\rangle^{-d}\|_{1,\infty}.
    \]
    Hence
    \[
        \limsup_{\lambda\to\infty} \Big|\Tr_\omega(\langle D\rangle^{-d})\frac{\Tr( f(P_{\lambda}AP_{\lambda}))}{\Tr(P_{\lambda})}-\Tr_{\omega}(f(A)\langle D\rangle^{-d})\Big| \leq 2\varepsilon \|\langle D\rangle^{-d}\|_{1,\infty}.
    \]
    Since $\varepsilon$ is arbitrary, this implies equation~\eqref{eq:Szego2}.
\end{proof}

We emphasise that Theorem~\ref{T:Szego} shows that the classical Szeg\H{o} theorems for determinants of Toeplitz matrices and Widom's generalisations thereof can be interpreted as properties of the noncommutative integral on spectral triples and their spectral truncations. 

\section{Fr\"{o}hlich functional}\label{S:Frohlich}
So far, we have considered situations modeled after $d$-dimensional spectral triples, where $\langle D\rangle^{-d} \in \mathcal{L}_{1,\infty}$. There are many examples of spectral triples that do not satisfy this condition, however. Instead, one could consider the property of $\theta$-summability, which says that $\Tr(e^{-tD^2}) < \infty$ for all $t>0$, or $\mathrm{Li}_1$-summability which requires $\Tr(e^{-t|D|}) < \infty$ for $t$ large enough.

For this section, we therefore assume that $D$ is a self-adjoint operator with compact resolvent, but we do not assume Weyl laws. Assuming $\mathrm{Li}_1$-summability, the functional
\[
a \mapsto \lim_{t\to \beta}\frac{\Tr(ae^{-t|D|})}{\Tr(e^{-t|D|})}, \quad a \in B(\Hc),
\]
which is sometimes called the Fr\"{o}hlich functional after \cite{FrohlichGrandjeanRecknagel1995,FrohlichChamseddineFelder1993}, has been studied extensively in the literature~\cite{GoffengMesland2018,GoffengRennie2019}. We highlight the relation between this functional and the one that has been the object of study in this note.

\begin{prop}
    Assume there exists $\beta \geq 0$ such that $\Tr(e^{-t|D|}) < \infty$ for $t > \beta$ and $\lim_{t \searrow \beta}\Tr(e^{-t|D|}) = \infty$. Then, for any extended limit $\omega \in \ell_\infty^*$ there exists an extended limit $\hat{\omega}_{D,\beta}\in \ell_\infty^*$ depending on $D$ and $\beta$ such that
    \begin{equation}\label{eq:FrohlichProp1}
    \omega \bigg(\frac{\Tr(ae^{- (\beta + \frac{1}{n}) |D|})}{\Tr(e^{-(\beta + \frac{1}{n})|D|})}  \bigg) = \hat{\omega}_{D,\beta} \bigg(\frac{\Tr(P_{\lambda_n} a P_{\lambda_n})}{\Tr(P_{\lambda_n})} \bigg), \quad a \in B(\Hc).
    \end{equation}
    Furthermore,
    \begin{equation}\label{eq:FrohlichProp2}
    \lim_{\lambda \to \infty} \frac{\Tr(P_\lambda a P_\lambda)}{\Tr(P_\lambda)} = \lim_{t\to \beta}\frac{\Tr(ae^{-t|D|})}{\Tr(e^{-t|D|})}, \quad a \in B(\Hc),
    \end{equation}
    in the sense that if the LHS limit exists, then the RHS limit exists and the equality holds.
\end{prop}
\begin{proof}
Write $\{r_k\}_{k=0}^\infty$ for the eigenvalues of $|D|$ counted \textit{without} multiplicity so that $r_0 < r_1 < \cdots$. Observe the identity $r_k = \lambda_{N(r_k)}$ where $N(\lambda):= \# \{k : \lambda_k \leq \lambda \}$ is the spectral counting function of $|D|$ and $\{\lambda_n\}_{n=0}^\infty$ are the eigenvalues of $|D|$ counted \textit{with} multiplicity. Then,
\begin{align*}
    \frac{\Tr(ae^{- (\beta + \frac{1}{n}) |D|})}{\Tr(e^{-(\beta + \frac{1}{n})|D|})} &= \frac{1}{\Tr(e^{-(\beta + \frac{1}{n})|D|})} \sum_{k=0}^\infty \big(\sum_{\lambda_j = r_k} \langle e_j, a e_j \rangle\big)e^{-(\beta + \frac{1}{n})r_k}\\
    &= \frac{1}{\Tr(e^{-(\beta + \frac{1}{n})|D|})} \sum_{k=0}^\infty \bigg(\Tr(P_{r_k})\frac{\Tr(P_{r_k}a P_{r_k})}{\Tr(P_{r_k})}- \Tr(P_{r_{k-1}}) \frac{\Tr(P_{r_{k-1}} a P_{r_{k-1}})}{\Tr(P_{r_{k-1}})} \bigg)   e^{-(\beta + \frac{1}{n})r_k}.
\end{align*}
Hence if we define $\hat{\omega}_{D,\beta} \in (\ell_\infty)^*$ by
\[
\hat{\omega}_{D,\beta}(b) := \omega \bigg( \frac{1}{\Tr(e^{-(\beta + \frac{1}{n})|D|})} \sum_{k=0}^\infty \Big(\Tr(P_{r_k}) b_{N(r_k)}  - \Tr(P_{r_{k-1}}) b_{N(r_{k-1})} \Big)   e^{-(\beta + \frac{1}{n})r_k}  \bigg), \quad b \in \ell_\infty,
\]
we have by construction that 
\[
    \omega \bigg(\frac{\Tr(ae^{-(\beta +  \frac{1}{n}) |D|})}{\Tr(e^{-(\beta + \frac{1}{n})|D|})}  \bigg) = \hat{\omega}_{D,\beta} \bigg(\frac{\Tr(P_{\lambda_n} a P_{\lambda_n})}{\Tr(P_{\lambda_n})} \bigg), \quad a \in B(\Hc),
\]
which is equation~\eqref{eq:FrohlichProp1}. Crucially, $\hat{\omega}_{D, \beta}$ is an extended limit if and only if $\lim_{t \searrow \beta}\Tr(e^{-t|D|}) = \infty$, see~\cite[Theorem~III.2]{Hardy1949}.

Equation~\eqref{eq:FrohlichProp2} is proved through the continuous version of the cited theorem, namely~\cite[Theorem~III.5]{Hardy1949}. If the limit 
\[
    \lim_{\lambda \to \infty} \frac{\Tr(P_\lambda a P_\lambda)}{\Tr(P_\lambda)}
\]
exists, then all extended limits on the sequence $\frac{\Tr(P_{\lambda_n} a P_{\lambda_n})}{\Tr(P_{\lambda_n})}$ coincide, and so we conclude
\[
    \lim_{\lambda \to \infty} \frac{\Tr(P_\lambda a P_\lambda)}{\Tr(P_\lambda)} = \lim_{t\to \beta}\frac{\Tr(ae^{-t|D|})}{\Tr(e^{-t|D|})}, \quad a \in B(\Hc).\qedhere
\]
\end{proof}

Writing $P_D := \chi_{[0,\infty)}$ and applying the above results to $P_D D$ instead of $D$, we have that
\[
\omega \bigg(\frac{\Tr(P_D ae^{- (\beta + \frac{1}{n})D})}{\Tr(P_D e^{-(\beta + \frac{1}{n})D})}  \bigg)= \hat{\omega}_{D,\beta} \bigg(\frac{\Tr(\chi_{[0,\lambda_n]}(D) a \chi_{[0,\lambda_n]}(D))}{\Tr(\chi_{[0,\lambda_n]}(D))} \bigg), \quad a \in B(\Hc),
\]
which is a functional that is extensively studied in~\cite{GoffengRennie2019}. In particular, it defines a KMS state of inverse temperature $\beta$ on the Toeplitz algebra generated by a $\mathrm{Li}_1$-summable spectral triple $(\A, \Hc, D)$ satisfying some extra conditions.

\section{Density of States}\label{S:DOS}
Another lens through which to interpret the discussed results so far is that of the density of states (DOS). In a sense, this perspective is the `Fourier transform' of the picture in the previous sections. So far, we have thought of $D$ as a Dirac-type operator, and $A \in B(\Hc)$ as a multiplication operator. We will now flip this around, taking $D$ a multiplication operator on $L_2(X)$ for some metric measure space $X$, and $A = f(H)$ for a potentially unbounded operator $H$ on $L_2(X)$.

Originating in solid state physics, the DOS describes for a quantum system, roughly speaking, how many quantum states are admitted at each energy level per unit volume. Usually, this framework is applied to study electrons in a solid material. For reviews of the DOS in mathematical physics, we refer to e.g.~\cite{PasturFigotin1992,Veselic2008,AizenmanWarzel2015}. 

Following Simon~\cite[Section C]{Simon1982}, we define the DOS as follows. Given a (possibly unbounded) self-adjoint operator $H$ on the Hilbert space $L_2(X)$ where $X$ is some metric space with a Borel measure written as $|\cdot|$, we consider the limits
\begin{equation*}
\lim_{R\to \infty} \frac{1}{|B(x_0,R)|}\mathrm{Tr}(f(H)M_{\chi_{B(x_0,R)}}), \quad f\in C_c(\R),
\end{equation*}
where $B(x_0,R)$ denotes the closed ball with center $x_0 \in X$ and radius $R$. When these limits exist (this includes assuming that $f(H)M_{\chi_{B(x_0,R)}}$ is trace-class), the limit is a positive continuous linear functional on $C_c(\R)$ and hence, via the Riesz--Markov--Kakutani theorem, we obtain a Borel measure $\nu_H$ on $\R$~\cite[Proposition C.7.2]{Simon1982} such that
\begin{equation}\label{eq:DOS}
\lim_{R\to \infty} \frac{1}{|B(x_0,R)|}\mathrm{Tr}(f(H)M_{\chi_{B(x_0,R)}}) = \int_{\mathbb{R}}f \,d\nu_H, \quad f\in C_c(\mathbb{R}).
\end{equation}
The measure $\nu_H$, if it exists, is what we call the density of states of the operator $H.$  

The main result of this section concerns a Dixmier trace formula for the density of states (DOS) on discrete metric spaces, which gives a variant of the main result of~\cite{AzamovHekkelman2022}. 
This is an equality
\[
\Tr_\omega(f(H)M_w) = C \int_\R f\, d\nu_H, \quad f\in C_c(\R),
\]
where $w: X \to \C$ is a weight such that $M_w \in \mathcal{L}_{1,\infty}$ and $\omega \in \ell_\infty^*$ is an extended limit. In~\cite{AMSZ}, this formula was proven for $X = \R^n$ and $H = -\Delta+M_V$ a Schr\"odinger operator. As explained there, on $\R^n$ with $H = -\Delta$, this formula is nothing but the Fourier transform of Connes' integration formula applied to a radial function (Theorem~\ref{T:ConnesIntegration}). In following works~\cite{AzamovHekkelman2022, HekkelmanMcDonald2024}, the formula was extended to certain discrete spaces and certain manifolds of bounded geometry, respectively.

Let $(X,d_X)$ be an infinite metric space and let $x_0 \in X.$ We assume that all metric balls contain finitely many points. Let $w:X\to \C$ be defined by 
\[
w(x) = \frac{1}{|B(x_0,d_X(x,x_0))|}
\]
where in this case $|\cdot|$ is the cardinality.
Note that the spectral projections $P_\lambda = \chi_{[0,\lambda]}(M_w^{-1})$ of this operator are multiplication operators $M_{\chi_{B(x_0,R_\lambda)}}$ for certain corresponding $R_\lambda \in \R_{>0}.$
Let $\{r_k\}_{k=0}^\infty$ be an increasing enumeration of the set $\{d(x,x_0)\;:\; x\in X\}.$ That is, $\{r_k\}_{k=0}^\infty$ lists the set of distances of points from $x_0$ without taking into account multiplicities.

\begin{prop}\label{P:PropCWeyl}
    Let $X$ and $w$ be as above. For the spectral counting function $N(\lambda)= \Tr(\chi_{[-\lambda,\lambda]}(M_w^{-1}))$, we have
    \[
    N(\lambda) \sim \lambda, \quad \lambda \to \infty \iff \frac{|B(x_0,r_{k+1})|}{|B(x_0, r_k)|} \to 1, \quad k \to \infty.
    \]
\end{prop}
\begin{proof}
    $\Longrightarrow$: For $\lambda \geq 1$, we have 
    \begin{align*}
        N(\lambda) &= \# \{x \in X \; : \; |B(x_0, d_X(x,x_0))| \leq \lambda \}\\
        &= |B(x_0, r_{i_\lambda})|,
    \end{align*}
    where $r_{i_\lambda}$ is the largest element in $\{r_k\}_{k \in \N}$ such that $|B(x_0,r_k)| \leq \lambda.$

    In particular, $N(|B(x_0,r_k)|) = |B(x_0,r_k)|$, and $N(|B(x_0,r_k)|-\frac{1}{2}) = |B(x_0,r_{k-1})|$. Hence, if $N(\lambda) \sim \lambda$, it follows that
    \[
     \frac{|B(x_0,r_{k+1})|-\frac{1}{2}}{|B(x_0,r_{k})|} = \frac{|B(x_0,r_{k+1})|-\frac12}{N(|B(x_0,r_{k+1})|-\frac12)} \to 1, \quad k \to \infty.
    \]
    Hence, 
    \[
    \frac{|B(x_0,r_{k+1})|}{|B(x_0,r_k)|} \to 1, \quad k \to \infty.
    \]
    $\Longleftarrow$: As before, we have 
    \[
    N(\lambda)=|B(x_0,r_{i_\lambda})|,
    \]
    and hence
    \[
    N(\lambda)=|B(x_0, r_{i_\lambda})| \leq \lambda \leq  |B(x_0,r_{1+i_\lambda})|.
    \]
    Dividing these inequalities by $N(\lambda)$ gives
    \[
    1 \leq \frac{\lambda}{N(\lambda)} \leq \frac{|B(x_0,r_{1+i_\lambda})|}{|B(x_0,r_{i_\lambda})|}.
    \]
    If $\frac{|B(x_0, r_{k+1})|}{|B(x_0,r_{k})|} \to 1$, it follows that
     $N(\lambda) \sim \lambda$ as $\lambda\to \infty$.
\end{proof}

We observe that Proposition~\ref{P:PropCWeyl} is closely related to~\cite[Proposition~2.9]{CiprianiSauvageot2021}.

\begin{thm}\label{T:NewDOS}
    Let $(X,d_X)$ and $w:X\to\C$ be as before, and suppose that \begin{equation}\label{E: Condition}
        \lim_{k\rightarrow \infty}\frac{|B(x_0, r_{k+1})|}{|B(x_0, r_k)|} = 1.
    \end{equation}
    Then for every extended limit $\omega \in \ell_\infty^*$ and bounded operator $T \in B(\ell_2(X))$,
    \begin{equation} \label{E: general DOS formula}
        \Tr_\omega(TM_w) =\omega \circ M\bigg( \frac{\Tr(TM_{\chi_{B(x_0,r_k)}})}{|B(x_0,r_k)|}\bigg).
    \end{equation}
\end{thm}
\begin{proof}
We have that $M_w\in \mathcal{L}_{1,\infty}$ (see~\cite[Lemma~4.1]{AzamovHekkelman2022}), and hence the left-hand side of equation~\eqref{E: general DOS formula} is well-defined. 

By Proposition~\ref{P:PropCWeyl} and condition~\eqref{E: Condition},
\[
N(\lambda)= \Tr(\chi_{[0,\lambda)}(M_w^{-1}))\sim \lambda, \quad \lambda \to \infty.
\]
The eigenvalues (without multiplicity) of $M_w^{-1}$ are equal to $\{|B(x_0, r_k)|\}_{k \in \N}$, where the eigenvalue $|B(x_0,r_k)|$ has multiplicity $|B(x_0,r_k)|-|B(x_0,r_{k-1})|$. Write $\{\lambda_k\}_{k\in \N}$ for an eigenvalue sequence of $M_w^{-1}$ without multiplicities, in increasing order. Given $k \in \N$, there exists $n_k \in \N$ such that
\[
|B(x_0, r_{n_k})| < k \leq|B(x_0,r_{1+n_k})|.
\]
Then, $\lambda_k = |B(x_0, r_{1+n_k})|$, and hence
\[
1 \leq \frac{\lambda_k}{k} < \frac{|B(x_0, r_{1+n_k})|}{|B(x_0, r_{n_k})|}.
\]
Condition~\eqref{E: Condition} now implies that
\[
\lambda_k \sim k, \quad k \to \infty.
\]
Hence, $M_w^{-1}$ satisfies Weyl's law, and Theorem~\ref{T:Log-CesaroMean} gives that for any extended limit $\omega \in \ell_\infty^*$,
\[
\frac{\Tr_\omega(TM_w)}{\Tr_\omega(M_w)} =\omega \circ M\bigg( \frac{\Tr(TM_{\chi_{B(x_0,r_k)}})}{|B(x_0,r_k)|}\bigg).
\]
Finally, since we proved that $\lambda_k \sim k$, it also follows that $\Tr_\omega(M_w) = 1$ independently of the extended limit $\omega \in \ell_\infty^*$ (this was also shown in~\cite[Corollary~5.2.3]{HekkelmanPhD}).
\end{proof}

Theorem~\ref{T:NewDOS} is less general than the main result of~\cite{AzamovHekkelman2022}, since the latter provides a result for a much larger class of weights $w: X \to \C$. However, for this particular choice of $w$, Theorem~\ref{T:NewDOS} is a stronger result than that was achieved in~\cite{AzamovHekkelman2022}, since it does not assume the existence of the limits
\[
\lim_{k\to \infty} \frac{1}{|B(x_0,r_k)|}\Tr(TM_{\chi_{B(x_0,r_k)}}).
\]
Furthermore, from Theorem~\ref{T:NewDOS} we now see that $\Tr_\omega(TM_w)$
is independent of the extended limit $\omega \in \ell_\infty^*$ if and only if
\[
\lim_{k\to\infty}M\bigg(\frac{1}{|B(x_0, r_k)|} \Tr(TM_{\chi_{B(x_0,r_k)}}) \bigg)
\]
exists, which is novel.

The Szeg\H{o} limit theorem from Section~\ref{S:Szego} takes on an entirely different role in this setting. In light of Lemma~\ref{L:Widom}, we obtain the following form of Theorem~\ref{T:Szego}.

\begin{cor}\label{C:DOSes}
    Let $X$ and $w$ as before, and assume that
    \begin{equation*}
        \lim_{k\rightarrow \infty}\frac{|B(x_0, r_{k+1})|}{|B(x_0, r_k)|} = 1.
    \end{equation*}
    If for some $\alpha > 0$, $H$ maps $\dom (M_w^{-\alpha})$ into $\dom (M_w^{-\alpha})$ and $[M_w^{-\alpha}, H]$ extends to a bounded operator on $\ell_2(X)$, then we have for all $f\in C_c(\R)$ and extended limits $\omega \in \ell_\infty^*$,
    \begin{equation*}
        \omega \circ M\bigg( \frac{\Tr(f(H)M_{\chi_{B(x_0,r_k)}})}{|B(x_0,r_k)|}\bigg) = \omega \circ M \bigg(\frac{\Tr(f(H|_{B(x_0,R)})M_{\chi_{B(x_0,r_k)}})}{|B(x_0,r_k)|} \bigg) = \Tr_\omega(f(H)M_w) ,
    \end{equation*}
    where $H|_{B(x_0,R)}:= M_{\chi_{B(x_0,r_k)}}HM_{\chi_{B(x_0,r_k)}}$.
\end{cor}
We see that the `Szeg\H{o} limit theorem' is now a result connecting  two differing ways of defining the DOS. Namely, an alternative definition of the DOS than via equation~\eqref{eq:DOS} is that via the limits
\[
\lim_{R\to \infty} \frac{1}{|B(x_0,R)|}\mathrm{Tr}(f(H|_{B(x_0,R)})M_{\chi_{B(x_0,R)}}) = \int_{\mathbb{R}}f \,d\tilde{\nu}_H, \quad f\in C_c(\mathbb{R}),
\]
and Corollary~\ref{C:DOSes} gives conditions for a Dixmier trace formula to hold for both approaches. In particular, Corollary~\ref{C:DOSes} also gives that under the listed conditions for $H$, we have that
\[
\lim_{k\to\infty}M\bigg(\frac{1}{|B(x_0, r_k)|} \Tr(f(H)M_{\chi_{B(x_0,r_k)}}) \bigg) =
\lim_{k\to\infty}M\bigg(\frac{1}{|B(x_0, r_k)|} \Tr(f(H|_{B(x_0,R)})M_{\chi_{B(x_0,r_k)}}) \bigg),
\]
in the sense that if one limit exists, the other limit exists and they are equal.

The condition that $[M_w^{-\alpha},H]$ extends to a bounded operator for some $\alpha >0$ is satisfied in most common situations. 
\begin{ex}
    Consider $(\Z^d,d_{\ell_1})$ where $d_{\ell_1}$ is the distance induced by the $\ell_1$-norm (i.e. the graph distance). 
    Let $\{e_m\}_{m\in \Z^d}$ be the canonical orthonormal basis of $\ell_2(\Z^d)$. Define the discrete Laplacian by
    \[
    \Delta: e_m \mapsto \sum_{d_{\ell_1}(k,m)=1} (e_m-e_k).
    \]
    Then for any bounded real-valued potential $V:\Z^d\to\R$, the Schr\"odinger operator $H:=-\Delta + M_V$ maps $\dom(M_w^{-\frac{1}{d}})$ into $\dom(M_w^{-\frac{1}{d}})$, and $[H, M_w^{-\frac{1}{d}}]$ extends to a bounded operator. This is easily seen from the fact that $\Delta = \sum_{k=1}^d (2I -S_k-S_k^*)$, where $S_1, \ldots, S_d$ are the shift operators on $\ell_2(\Z^d)$, and $|B(0,k)|^{\frac{1}{d}} \sim  C_dk$.
\end{ex}

\section{Noncommutative ergodicity}\label{S:Ergodicity}
Quantum ergodicity began as a study of geodesic flow on manifolds through abstract operator theoretical language. On a closed Riemannian manifold $(M,g)$ we can define the geodesic flow as a map $G_t^M: SM \to SM$, where $SM$ is the unit sphere in the tangent bundle of the manifold $M$. For a point $(x,v) \in SM$, one simply takes the unique geodesic $\gamma: \R \to M$ with $\gamma(0) = x$ and $\gamma'(0) = v$, and defines $G_t^M(x,v):= (\gamma(t), \gamma'(t))$. This flow is said to be ergodic if every measurable function $f \in L_\infty(SM)$ which is fixed by the flow (i.e. $f \circ G_t^M = f$ almost everywhere) is constant almost everywhere. Equivalently, the geodesic flow can be defined on $S^*M$, the unit sphere in the cotangent bundle. 

Let $\{e_k\}_{k=0}^\infty$ be any orthonormal basis of eigenvectors of the Laplace--Beltrami operator $\Delta_g$, and let $P_\lambda := \chi_{[0,\lambda]}(\Delta_g)$. Related to the result derived in Section~\ref{S:Integration}, it is known~\cite[Section~4]{ColindeVerdiere1985} that for compact Riemannian manifolds we have that
\[
\frac{\Tr(P_\lambda \mathrm{Op}(a) P_\lambda)}{\Tr(P_\lambda)} \to \int_{S^*M}a \, d\nu,
\]
where $a \in C^\infty(S^*M)$ and $\mathrm{Op}(a)$ is a classical pseudodifferential operator with principal symbol~$a$.
Shnirelman, Zelditch, and Colin de Verdi\`ere showed~\cite{Shnirelman1974,ColindeVerdiere1985,Zelditch1987} that this fact can be used, if $M$ has ergodic geodesic flow, to show that there exists a density one subsequence $\{e_{j}\}_{j \in J}$ of $\{e_k\}_{k=0}^\infty$, meaning that $\frac{\#(J \cap \{0,\ldots, n\})}{n+1} \to 1$, such that
\[
\lim_{J\ni j\to\infty} \langle e_{j}, \mathrm{Op}(a) e_{j}\rangle =\int_{S^*M}a \, d\nu.
\]
This and related properties are called quantum ergodicity of the operator $\Delta_g$.

Before we start to put quantum ergodicity results into a noncommutative geometrical context, let us observe first that our labours in Section~\ref{S:Integration} provide a result in the other direction. The Weyl measure of an operator, which is the relevant measure for quantum ergodicity~\cite[Section~4]{ColindeVerdiereHillairet2018}, admits a Dixmier trace formula.

\begin{defn}\label{D:WeylMeasure}
    Let $M$ be a manifold equipped with a nonvanishing density $\rho$, and let $\Delta$ be a self-adjoint positive operator on $L_2(M,\rho)$ with compact resolvent. Let $\{e_k\}_{k=0}^\infty$ be an orthonormal basis of $L_2(M,\rho)$ consisting of eigenvectors of $\Delta$ with corresponding eigenvalues $\{\lambda_k\}_{k=0}^\infty$. If
\[
\lim_{\lambda \to \infty} \frac{1}{N(\lambda)} \sum_{\lambda_k \leq \lambda} \langle  e_k, M_f e_k \rangle
\]
exists for all $f \in C_c(M)$, then there exists a measure $\mu_\Delta$ such that
\[
\lim_{\lambda \to \infty} \frac{1}{N(\lambda)} \sum_{\lambda_k \leq \lambda} \langle  e_k, M_f e_k \rangle = \int_M f \, d\mu_\Delta.
\]
This measure is called the \textit{local Weyl measure} of $\Delta$. 
\end{defn}

\begin{prop}\label{P:DixTraceWeyl}
    If $\Delta$ as in Definition~\ref{D:WeylMeasure} satisfies Weyl's law
    \[
    \lambda(k, \Delta) \sim C k^\frac{2}{d}
    \]
    for some $0 < d \in \R$ and admits a local Weyl measure $\mu_\Delta$, then
    \begin{equation}\label{eq:DixTraceWeyl}
        \Tr_\omega(M_f (1+\Delta)^{-\frac{d}{2}}) =  \Tr_\omega((1+\Delta)^{-\frac{d}{2}}) \int_M f \, d\mu_\Delta.    
    \end{equation}
\end{prop}
\begin{proof}
    Consequence of Theorem~\ref{T:Log-CesaroMean}.
\end{proof}

This is relevant for sub-Riemannian manifolds, in which case one can take $\Delta$ to be the sub-Laplacian and $\mu_\Delta$ is not necessarily the usual volume form on the manifold $M$. Notably, a rescaling of this measure was found very recently in~\cite[Section~1.4]{KordyukovSukochev2024} to be a spectrally correct sub-Riemannian volume of $M$, additionally providing in that context a generalisation of the above Dixmier trace formula to any normalised continuous trace $\phi$. This measure is studied extensively in this context in~\cite{ColindeVerdiereHillairet2018} as well.

We will now shift our attention to results in quantum ergodicity which are interesting when viewed from the perspective of noncommutative geometry. To start, we provide an analogue of ergodicity of the geodesic flow -- a property a compact Riemannian manifold can have, which we should therefore be able to see as a property of a spectral triple. For this purpose we recall the following construction and theorem by Connes~\cite[Section~6]{Connes1995}. For $A \in B(\Hc)$, we write $\sigma_t(A):= e^{it|D|}Ae^{-it|D|}$, $t \in \R$.

\begin{thm}\label{T:ConnesS*A}
    For a unital regular spectral triple $(\A, \Hc, D)$, where `regular' means that that $\delta^n(a) \in B(\Hc)$ for all $a \in \A$, $n \in \mathbb{N}$, define 
    \[
    S^*\A := C^*\bigg(\bigcup_{t\in \R}\sigma_t\big(\A\big) + K(\Hc)\bigg) \big/ K(\Hc).
    \]
    This $C^*$-algebra comes equipped with automorphisms 
    \[
    G_t(A + K(\Hc)):= e^{it|D|} A e^{-it|D|} + K(\Hc).
    \]
    For $(\A, \Hc, D) \simeq (C^\infty(M), L_2(S), D_M)$, the Dirac spectral triple associated to a compact Riemannian spin manifold, we have $S^*\A \simeq C(S^*M)$. 
    Furthermore, Egorov's theorem implies that the action of $G_t$ on $C(S^*M)$ is given by the geodesic flow $G_t^M$, see e.g~\cite[Section~9.2]{Zelditch2017}\cite[Section~11.1]{Zworski2012}.
\end{thm}

The fact that $G_t$ is given by the geodesic flow $G_t^M$ in the commutative setting, provides the basis for interpreting $G_t$ as an analogue of geodesic flow even in the noncommutative case.
A few examples of the construction $S^*\A$ are given in~\cite{GolseLeichtnam1998}. In the context of foliations of manifolds, it has been covered in~\cite{Kordyukov2005}.

\begin{rem}
    In the original formulation of Theorem~\ref{T:ConnesS*A} in~\cite{Connes1995}, $S^*\A$ was instead constructed as the space
    \[
    S^*\A := C^*\bigg(\bigcup_{t\in \R}\sigma_t\big(\Psi^0\big) + K(\Hc)\bigg) \big/ K(\Hc),
    \]
    where for a spectral triple $(\A, \Hc, D)$ Connes writes $\Psi^0$ for the set of operators admitting an asymptotic expansion
    \[
    P = b_0 + b_{-1}\langle D\rangle^{-1} + b_{-2}  \langle D\rangle^{-2} + \cdots, \quad b_j \in \mathcal{B},
    \]
    with $\mathcal{B}$ generated by $\A$ and $\delta^n(\A)$, where $\delta(a):= [|D|,a]$. The asymptotic expansion means that the difference between $P$ and the $n$th partial summand extends to a bounded linear operator from $\dom(\langle D\rangle^{s})$ to $\dom (\langle D\rangle^{s+n})$ for every $s\in \mathbb{R}.$
    Note that the operators $[D,\A]$ are not included in~$\mathcal{B}$. 

    For a unital spectral triple, $\langle D\rangle^{-1}$ is compact. And since for $b\in \mB,$ the second commutator $[|D|,[|D|,b]]$ is bounded, we have norm convergence
    \[
    \lim_{t\to 0} \frac{\sigma_t(b)-b}{t} = i[|D|,b],
    \]
    and hence Connes' original construction and the one in Theorem~\ref{T:ConnesS*A} (also used in~\cite{GolseLeichtnam1998}) are the same. 

    Note that it is important that the operators $[D,\A]$ are not included in $\Psi^0$. For illustration, in the commutative case $|D|$ acts with scalar principal symbol on the vector bundle $S$, meaning that $\mathcal{B}$ and hence $\Psi^0$ can be regarded as acting on $L_2(M)$ instead of $L_2(S)$. The isomorphism $S^*C^\infty(M) \cong C(S^*M)$ in Theorem~\ref{T:ConnesS*A} is then simply an extension of the symbol map $\Psi^0_{cl} \to C^\infty(S^*M)$ on classical pseudodifferential operators on $M$.
\end{rem}

The automorphisms $G_t$ provide an action of $\R$ on the $C^*$-algebra $S^*\A$, and this noncommutative cotangent sphere is thus an example of a $C^*$-dynamical system.

\begin{defn}
    A $C^*$-dynamical system $(\A, G, \alpha)$ consists of a $C^*$-algebra $\A$, a locally compact group $G$, and a strongly continuous representation $\alpha: G \to \operatorname{Aut}(\A)$. 
\end{defn}

There is a vast literature on $C^*$-dynamical systems, see~\cite[Section~2.7]{BratteliRobinson1987} for a start. In particular it has been a popular object of study in the field of quantum ergodicity, see e.g.~\cite{Zelditch1996}.

Recall that for compact manifolds, geodesic flow is said to be ergodic if the only measurable functions that are invariant almost everywhere under the geodesic flow are the functions that are constant almost everywhere. This definition is measure-theoretic in nature, and to translate it into a statement on spectral triples we therefore define the noncommutative $L_2$-space on $S^*\A$, which corresponds with $L_2(S^*M)$ in the commutative case.

\begin{prop}
    Let $(\A, \Hc, D)$ be a unital regular spectral triple where $D^2$ satisfies Weyl's law (Definition~\ref{D:WeylLaw}). The functional
    \[
    \tau(A + K(\Hc)):= \frac{\Tr_\omega(A\langle D\rangle^{-d})}{\Tr_\omega(\langle D \rangle^{-d})}, 
    \]
    defines a finite positive trace on $S^*\A$.
\end{prop}
\begin{proof}   
This is a standard result, see e.g.~\cite[Theorem~6.1]{CiprianiSauvageot2021} for a general formulation. We remark that the traciality of $\tau$ in fact follows from Theorem~\ref{T:Log-CesaroMean}, Widom's Lemma~\ref{L:Widom}, and the trivial identity
\[
\Tr(P_\lambda A P_\lambda B P_\lambda) = \Tr(P_\lambda B P_\lambda A P_\lambda),
\]
which is a novel proof of this fact.
\end{proof}

\begin{defn}
    We define $L_2(S^*\A)$ as the Hilbert space $\Hc_\tau$ in the GNS construction $(\pi_\tau, \Hc_\tau)$. Explicitly, writing $I= \{A + K(\Hc) \in S^*\A : \tau(A^*A) = 0 \}$, we define
    \[
    L_2(S^*\A) := \overline{S^*\A / I}^{\|\cdot\|_{L_2}},
    \]
    where the completion is taken in the semi-norm $\| A + I \|_{L_2} = \big(\tau(A^*A)\big)^{\frac{1}{2}}$. The space $L_2(S^*\A)$ is a Hilbert space with inner product defined via
    \[
    \langle A + I, B+I \rangle_{L_2} := \tau(B^*A), \quad A,B \in S^*\A.
    \]
    \end{defn}
Observe that the automorphism $G_t$ on $S^*\A$ extends to a unitary operator $G_t \in B(L_2(S^*\A))$.

\begin{nota}
    In accordance with the paradigm called the $C^*$-algebraic approach to the principal symbol~\cite{Cordes1979, Dao1, Dao2, Dao3}, we write
    \[
     \sym :C^*\bigg(\bigcup_{t\in \R}\sigma_t\big(\A\big) + K(\Hc)\bigg) \to S^*\A
    \]
    for the defining quotient map of $S^*\A$ (Theorem~\ref{T:ConnesS*A}), which is understood as a symbol map. Writing $\pi$ for the quotient map
    \begin{align*}
        \pi: S^*\A &\to L_2(S^*\A)\\
        A &\mapsto A+ I',
    \end{align*}
    we will furthermore use the notation 
    \[
    \syml := \pi \circ \sym :C^*\bigg(\bigcup_{t\in \R}\sigma_t\big(\A\big) + K(\Hc)\bigg) \to L_2(S^*\A).
    \]
\end{nota}

\begin{ex}\label{ex:L2SAspaces}
    \begin{enumerate}
        \item \label{ex:commutativel2original} For the Dirac spectral triple coming from a compact Riemannian spin manifold, \\ $(C^\infty(M), L_2(S), D_M)$, we have that  $S^*C^\infty(M) \simeq C(S^*M)$ with $\tau_{S^*\A} = \int_{S^*M}$. Hence $L_2(S^*C^\infty(M)) \simeq L_2(S^*M)$. The action $G_t$ agrees with the usual geodesic flow.
        \item \label{Ex:AlmostComm} Given an even dimensional compact Riemannian manifold, and a finite dimensional spectral triple $(\A_F, \Hc_F, D_F)$, we have for the almost commutative manifold (see Section~\ref{S:Preliminaries}) 
        \[
        (\A := C^\infty(M) \otimes \A_F, L_2(S)\otimes \Hc_F, D:=D_M \otimes 1 + \gamma_M \otimes D_F),
        \]
        that 
        $S^*\A \simeq C(S^*M)\otimes \A_F$ with $\tau_{S^*\A} = \int_{S^*M} \otimes \Tr$. Hence $L_2(S^*\A) \simeq L_2(S^*M)\otimes HS_F$, where $HS_F$ is simply $\A_F$ equipped with the Hilbert--Schmidt norm. The automorphisms $G_t$ act as $G_t^M \otimes 1$, where $G_t^M$ is the usual geodesic flow on $S^*M$. This corrects~\cite[Lemma~2.2]{GolseLeichtnam1998}.
        \item For the noncommutative torus $(C^\infty(\mathbb{T}^d_\theta), L_2(\mathbb{T}^d_\theta) \otimes \mathbb{C}^{N_d}, D)$ (see Section~\ref{S:Preliminaries}), we have that $S^*C(\mathbb{T}_\theta^d) \simeq C(\mathbb{T}^d_\theta) \otimes C(\mathbb{S}^{d-1})$ with $\tau_{S^*\mathbb{T}^d_\theta} = \tau_{\mathbb{T}^d_\theta} \otimes \int_{\mathbb{S}^{d-1}}$ and $\otimes$ is the minimal $C^*$-tensor product. Hence $L_2(S^* \mathbb{T}^d_\theta) \simeq  L_2(\mathbb{T}^d_\theta) \otimes L_2(\mathbb{S}^{d-1})$. The automorphisms $G_t$ act as
        \[
        G_t(u_{n} \otimes g) = u_n \otimes e_{t,n}g, \quad t\in\R, n \in \mathbb{Z}^d, g \in C(\mathbb{S}^{d-1}),
        \]
        where
        \[
        e_{t,n}(x) := \exp(it\, n \cdot x) , \quad t\in \R,n \in \mathbb{Z}^d, x\in \mathbb{S}^{d-1}\subseteq \mathbb{R}^d.
        \]
        \item Let $\A$ be the Toeplitz algebra, i.e. the $C^*$-algebra generated by the shift operator on $\ell_2(\mathbb{N})$, and let $D$ be the operator on $\ell_2(\mathbb{N})$ defined on the standard basis $\{ e_j \}_{j \in \mathbb{N}}$
        \[
        D: e_j \mapsto j e_j, \quad j \in \mathbb{N}.
        \]
        For the spectral triple $(\A, \ell_2(\mathbb{N}), D)$, we have $S^*\A \simeq C(\mathbb{S}^1)$ with $\tau_{S^*\A} = \int_{\mathbb{S}^1}$. Hence $L_2(S^*\A) \simeq L_2(\mathbb{S}^1)$. The automorphism $G_t$ is given by rotation. 
        \end{enumerate}
\end{ex}
\begin{proof}
    $(1)$ can be found in~\cite[Proposition~2]{Connes1995}.\\
    $(2)$: Since $|D| = \sqrt{D_M^2\otimes 1 + 1 \otimes D_F^2}$, it follows that $|D| - |D_M|\otimes 1$ is a compact operator on $L_2(S) \otimes \Hc_F$. We will show this with a double operator integral argument. First, one can omit the kernels of $|D|$ and $|D_M|\otimes 1$ from the Hilbert space as the projection onto the kernel of either operator is finite-rank and thus compact. Both operators have compact resolvent. Hence, after this modification, the function $f(x) = \sqrt{x}$ is smooth on a neighbourhood of the spectra of the operators  $|D|$ and $|D_M|\otimes 1$. Define the Sobolev spaces $\Hc^s := \dom |D_M|^s \otimes 1$, $s \in \R$, and apply~\cite[Theorem~6]{MOOIs} to find that
\[
|D| - |D_M|\otimes 1 = T^{D^2, D_M^2 \otimes 1}_{f^{[1]}}(1 \otimes D_F^2) \in K(\Hc),
\]
where $T^{D^2, D_M^2 \otimes 1}_{f^{[1]}}(1 \otimes D_F^2)$ is a double operator integral. Its compactness is a consequence of the fact that this multiple operator integral is a negative order pseudodifferential operator --- to be precise, it is an element of $ \op^{-1+\varepsilon}(|D_M|\otimes 1)$ for any $\varepsilon > 0$ in the notation of~\cite{MOOIs}.
It now follows from Duhamel's formula that
\[
e^{it|D|} - e^{it(|D_M|\otimes 1)} = it \int_0^1 e^{ist|D|} (|D| - |D_M| \otimes 1) e^{i(1-s)t(|D_M|\otimes 1)} \, ds  \in K(\Hc).
\]
Therefore,
\[
\bigcup_{t\in \R}\sigma_t(\mathcal{B}) + K(\Hc) = \bigcup_{t\in \R}\sigma^M_t(\mathcal{B}_M) \otimes \A_F + K(\Hc),
\]
where $\sigma_t^M$ is the conjugation with $e^{it|D_M|}$ on $M$ and $\mB_M$ the algebra generated by $\delta_{|D_M|}^n(C^\infty(M)) \subseteq B(L_2(M))$. We conclude that $S^*\A \simeq C(S^*M)\otimes \A_F$. This proof shows that the action $G_t$ on $ C(S^*M)\otimes \A_F  $ is given by $G_t = G_t^M \otimes 1$.
      \\   
    $(3)$: Although the Hilbert space of the spectral triple is $L_2(\mathbb{T}^d_\theta) \otimes \mathbb{C}^{N_d}$, since $|D| = \sqrt{-\Delta} \otimes 1$ (see Section~\ref{S:Preliminaries}) similarly to the manifold case $\mB$ acts trivially on the $\mathbb{C}^{N_d}$-component. We can therefore make the identification $\mB \subseteq B(L_2(\mathbb{T}^d_\theta))$. In fact, we claim that $\overline{\mB}^{\|\cdot\|}$ is a $C^*$-algebra stable under the action $\sigma_t(\cdot) = e^{it\sqrt{-\Delta}} (\cdot) e^{-it\sqrt{-\Delta}}$, and therefore
    \begin{equation}\label{eq:torusS*A1}
    S^*C(\mathbb{T}^d_\theta) \simeq (\overline{\mB}^{\|\cdot\|} + K(L_2(\mathbb{T}^d_\theta)))/K(L_2(\mathbb{T}^d_\theta)).
    \end{equation}
    The claim holds since formally $\sigma_t(a) = \sum_{k=0}^\infty \frac{(it)^k}{k!} \delta^k(a)$, and this sum is actually norm convergent for     
    $a \in \mathrm{Poly}(\mathbb{T}^d_\theta) := \mathrm{span}\{u_n\}_{n\in \mathbb{Z}^d}$. Denoting the generated $*$-algebra $\mB_{poly}:= \langle a, \delta^n(a)\rangle_{a \in \mathrm{Poly}(\mathbb{T}^d_\theta)}$ we therefore have
    \[
    \sigma_t : \mB_{poly} \to \overline{\mB}^{\|\cdot\|}.
    \]
    Since $ \mathrm{Poly}(\mathbb{T}^d_\theta)$ is dense in $C^\infty(\mathbb{T}^d_\theta)$ and $\sigma_t$ is an isometry on $B(L_2(\mathbb{T}^d_\theta))$, it is easily seen that this implies that $\sigma_t$ maps $\overline{\mB}^{\|\cdot\|}$ into itself, proving~\eqref{eq:torusS*A1}.
    
    By construction $C(\mathbb{T}^d_\theta)$ is represented on $L_2(\mathbb{T}^d_\theta)$ as bounded left-multiplication operators (denote the representation $\pi_1$), and $C(\mathbb{S}^{d-1})$ is as well via the representation
    \[
    \pi_2(g) = g\big(\frac{D_1}{\sqrt{-\Delta}}, \ldots, \frac{D_d}{\sqrt{-\Delta}} \big),\quad g \in C(\mathbb{S}^{d-1}),
    \]
    where $D_i: u_k \mapsto k_i u_k$. It is shown in~\cite{Dao2} that, writing $\Pi(C(\mathbb{T}^{d}_\theta), C(\mathbb{S}^{d-1}) )$ for the $C^*$ -algebra generated by $\pi_1( C(\mathbb{T}^{d}_\theta))$ and $\pi_2( C(\mathbb{S}^{d-1}))$ inside $B(L_2(\mathbb{T}^d_\theta))$, we have
    \begin{equation}\label{eq:DaoeqTorus}
    \Pi(C(\mathbb{T}^{d}_\theta), C(\mathbb{S}^{d-1}) ) / K(L_2(\mathbb{T}^d_\theta)) \simeq C(\mathbb{T}^{d}_\theta) \otimes C(\mathbb{S}^{d-1}).
    \end{equation}
    Comparing~\eqref{eq:torusS*A1} and~\eqref{eq:DaoeqTorus}, to determine that $S^*\mathbb{T}^d_\theta \simeq C(\mathbb{T}^d_\theta) \otimes C(\mathbb{S}^{d-1})$, it therefore suffices to show that 
    \begin{equation}\label{eq:TorusEquality}
    \overline{\mB}^{\|\cdot\|} + K(L_2(\mathbb{T}^d_\theta))= \Pi(C(\mathbb{T}^{d}_\theta), C(\mathbb{S}^{d-1}) ) + K(L_2(\mathbb{T}^d_\theta)) \subseteq B(L_2(\mathbb{T}^d_\theta)).
    \end{equation}
    To start, it is immediately obvious that $\pi_1(C(\mathbb{T}^{d}_\theta)) \subseteq \overline{\mB}^{\|\cdot\|}$. Next, the operators $\frac{D_i}{\sqrt{-\Delta}}$ generate $\pi_2(C(\mathbb{S}^{d-1}))$ as a $C^*$-algebra, and we claim that
    \begin{equation}\label{eq:ClaimTorus}
    u_{e_j}^*[\sqrt{-\Delta}, u_{e_j}] - \frac{D_j}{\sqrt{-\Delta}} \in K(L_2(\mathbb{T}^d_\theta)),
    \end{equation}
    where $e_j \in \mathbb{Z}^d$ is the standard unit vector.
    This would imply
    \begin{equation}\label{eq:TorusInclusion1}
     \Pi(C(\mathbb{T}^{d}_\theta), C(\mathbb{S}^{d-1}) ) + K(L_2(\mathbb{T}^d_\theta)) \subseteq \overline{\mB}^{\|\cdot\|} + K(L_2(\mathbb{T}^d_\theta)).
    \end{equation}
    Equation~\eqref{eq:ClaimTorus} is proven by writing 
    \begin{align*}
        \bigg(u_{e_j}^*[\sqrt{-\Delta}, u_{e_j}] - \frac{D_j}{\sqrt{-\Delta}}\bigg)u_k &= \bigg(|k+e_j| - |k| - \frac{k_j}{|k|} \bigg) u_{k}.        
    \end{align*}
    Now define 
    \[
    f(t, k):= |k+ t e_j|, \quad k\in \mathbb{Z}^d, t \in \R,
    \]
    and note that its derivatives in the $t$ variable are
    \[
    f'(t,k) = \frac{k_j+t}{|k+te_j|}, \quad f''(t,k) = \frac{|k+te_j|^2 - (k_j+t)^2}{|k+te_j|^3}.
    \]
    Hence 
    \begin{align*}
        \bigg(u_{e_j}^*[\sqrt{-\Delta}, u_{e_j}] - \frac{D_j}{\sqrt{-\Delta}}\bigg)u_k &= (f(1,k) - f(0,k) - f'(0,k))u_k\\
        &= \int_{0}^1 (1-t) f''(t, k) \, dt \cdot u_k.
    \end{align*}
    From the form of $f''(t,k)$ above, we therefore have
    \[
    |f(1,k) - f(0,k) - f'(0,k)| \in c_0(\mathbb{Z}^d),
    \]
    which indeed shows that $u_{e_j}^*[\sqrt{-\Delta}, u_{e_j}] - \frac{D_j}{\sqrt{-\Delta}}$ is a compact operator, proving~\eqref{eq:ClaimTorus} and therefore~\eqref{eq:TorusInclusion1}.

    For the other inclusion,
    \begin{equation}\label{eq:TorusInclusion2}
      \overline{\mB}^{\|\cdot\|} + K(L_2(\mathbb{T}^d_\theta)) \subseteq \Pi(C(\mathbb{T}^{d}_\theta), C(\mathbb{S}^{d-1}) ) + K(L_2(\mathbb{T}^d_\theta)),
    \end{equation}
    the above arguments already show that
    \[
    [\sqrt{-\Delta}, u_{e_j}] \in \Pi(C(\mathbb{T}^{d}_\theta), C(\mathbb{S}^{d-1}) ) + K(L_2(\mathbb{T}^d_\theta)).
    \]
    Since by explicit computation
    \[
    u_{e_j}^*\delta_{\sqrt{-\Delta}}^n(u_{e_j}) = (u_{e_j}^*[\sqrt{-\Delta}, u_{e_j}])^n, 
    \]
    and since $u_{e_j}^m = u_{m e_j}$, we have that
    \[
    \delta^n_{\sqrt{-\Delta}}(u_k) \in \Pi(C(\mathbb{T}^{d}_\theta), C(\mathbb{S}^{d-1}) ) + K(L_2(\mathbb{T}^d_\theta)), \quad n \in \mathbb{Z}_{\geq 0}, k \in \mathbb{Z}^d,
    \]
    and hence proving~\eqref{eq:TorusInclusion2}. The two inclusions~\eqref{eq:TorusInclusion1} and~\eqref{eq:TorusInclusion2} give the equality~\eqref{eq:TorusEquality}, which was noted to imply that
    $S^*\mathbb{T}^d_\theta \simeq C(\mathbb{T}^d_\theta) \otimes C(\mathbb{S}^{d-1})$. For the automorphism $G_t$, first note that
    \[
    \sigma_t(g(\frac{D_1}{\sqrt{-\Delta}}, \ldots, \frac{D_d}{\sqrt{-\Delta}})) = g(\frac{D_1}{\sqrt{-\Delta}}, \ldots, \frac{D_d}{\sqrt{-\Delta}}), \quad g\in C(\mathbb{S}^{d-1}).
    \]
    Next,
    \begin{align*}
        e^{it\sqrt{-\Delta}} u_{e_j}e^{-it\sqrt{-\Delta}} u_k &= e^{it(|k+e_j| - |k|)} u_{e_j} u_k\\
        &=u_{e_j} \exp(it u_{e_j}^* [\sqrt{-\Delta},u_{e_j}])  u_k\\
        &=u_{e_j}\exp\big(it \frac{D_j}{\sqrt{-\Delta}}\big) u_k +u_{e_j} \bigg(\exp(it u_{e_j}^* [\sqrt{-\Delta},u_{e_j}])  - \exp\big(it \frac{D_j}{\sqrt{-\Delta}}\big) \bigg) u_k.
    \end{align*}
    We have already seen that $  u_{e_j}^* [\sqrt{-\Delta},u_{e_j}]  -  \frac{D_j}{\sqrt{-\Delta}} \in K(L_2(\mathbb{T}^d_\theta))$, and hence as in the proof of Example~\ref{ex:L2SAspaces}.\ref{Ex:AlmostComm} it follows from Duhamel's formula that
    \begin{align*}
    \exp(it u_{e_j}^* [\sqrt{-\Delta},u_{e_j}]) - \exp\big(it \frac{D_j}{\sqrt{-\Delta}}\big) 
    \in K(L_2(\mathbb{T}^d_\theta)).
    \end{align*}
    Thus, we see that
    \[
        G_t(u_{n} \otimes g) = u_n \otimes e_{t,n}g, \quad t\in\R, n \in \mathbb{Z}^d, x\in \mathbb{S}^{d-1}\subseteq \mathbb{R}^d, g \in C(\mathbb{S}^{d-1}),
    \]
    where
    \[
        e_{t,n}(x) := \exp(it\, n \cdot x) , \quad t\in \R,n \in \mathbb{Z}^d, x\in \mathbb{S}^{d-1}\subseteq \mathbb{R}^d.
    \]
    $(4)$: It is well-known that, after identifying $\ell_2(\mathbb{N})$ with the Hardy space $H^2$, any element in the Toeplitz algebra $\A$ can be written as $T_\phi + K$, where $T_\phi$ is the Toeplitz operator with symbol $\phi \in C(\mathbb{S}^1)$ and $K \in K(\ell_2(\mathbb{N}))$, see e.g.~\cite[Section~3.5]{Murphy1990}. By an explicit computation, it can be seen that
    \[
    e^{it|D|} T_\phi e^{-it|D|} = T_{\phi \circ R_t},
    \]
    where $R_t$ is rotation by the angle $t$. Hence $\sigma_t(\A) = \A$, and 
    \[
    S^*\A = \A / K(\ell_2(\mathbb{N})) \simeq C(\mathbb{S}^1).
    \]
    For the noncommutative integral, we can use the diagonal formula in Theorem~\ref{T:Log-CesaroMean}, so that for an arbitrary element $T_\phi + K \in \A$,
    \[
    \Tr_\omega((T_\phi + K)\langle D\rangle^{-1}) = \omega \circ M (\langle e_k, (T_\phi + K) e_k \rangle ) = \int_{\mathbb{S}^1} \phi(t) \, dt. \qedhere
    \]
\end{proof}

\begin{defn}\label{def:ClassicErgodicity}
    We say that $(\A, \Hc, D)$ is classically ergodic if for $a \in L_2(S^*\A)$, we have $G_t(a) =a$ for all $t\in \R$ if and only if $a = \lambda \cdot 1 \in L_2(S^*\A)$ for some $\lambda \in \C$.
\end{defn}

The construction of $L_2(S^*\A)$ has now reached its goal; for spectral triples derived from compact Riemannian manifolds, this definition is precisely the usual definition of ergodicity of the geodesic flow (Example~\ref{ex:L2SAspaces2}.\ref{ex:commutativel2}).

We now immediately claim the following theorem, the NCG analogue of the classic result in quantum ergodicity by Shnirelman, Zelditch, and Colin de Verdi\`ere~\cite{Shnirelman1974, ColindeVerdiere1985, Zelditch1987}.
\begin{thm}\label{T:MainQE}
    Let  $(\A, \Hc, D)$ be a unital regular spectral triple with local Weyl laws. Assume that the closure of $\A$ in $B(\Hc)$ is separable. If the triple is classically ergodic, then for every basis $\{e_n\}_{n=0}^\infty$ of eigenvectors of $|D|$ there exists a density one subset $J \subseteq \mathbb{N}$ such that
    \[
    \lim_{J\ni j\to \infty}\langle e_j, A e_j \rangle = \frac{\Tr_{\omega}(A \langle D\rangle^{-d})}{\Tr_{\omega}(\langle D\rangle^{-d})}, \quad A \in \A.
    \]
\end{thm}
\begin{proof}
    Classical ergodicity of $(\A, \Hc,D)$ means that the projection onto the $G_t$-invariant vectors in $L_2(S^*\A)$ has rank $1$, which is called `uniqueness of the vacuum state' for the $C^*$-dynamical system $(S^*\A, \mathbb{R}, G_t)$ in~\cite{Zelditch1996}.  Hence, due to Proposition~\ref{P:HeattoInt}, the theorem is a consequence of~\cite[Lemma~2.1]{Zelditch1996}.
\end{proof}

This theorem, while its mathematical core is already an established result in quantum ergodicity, gives a fresh perspective on the criterion of a $C^*$-dynamical system having a `unique vacuum state'. And while the vast majority of results in the paper~\cite{Zelditch1996} are formulated for `quantised abelian' $C^*$-dynamical systems, which in our case would mean $S^*\A$ is represented as a commutative algebra on $L_2(S^*\A)$, the philosophy of noncommutative geometry provides solid reason to study not quantised abelian $C^*$-dynamical systems but ones with a unique vacuum state, as proposed by Zelditch~\cite{Zelditch1996}.

\begin{ex}\label{ex:L2SAspaces2}We continue Example~\ref{ex:L2SAspaces}.
    \begin{enumerate}
        \item \label{ex:commutativel2}The canonical spectral triple corresponding to a compact Riemannian spin manifold, $(C^\infty(M), L_2(S), D_M)$ is classically ergodic if and only if $M$ has ergodic geodesic flow.
        \item \label{ex:AlmostComm2}Any nontrivial almost commutative manifold $( C^\infty(M) \otimes \A_F, L_2(S)\otimes \Hc_F, D_M \otimes 1 + \gamma_M \otimes D_F)$ is not classically ergodic. Note that this corrects~\cite[Corollary~(3.1)]{Zelditch1996}, which was already known to experts to be false.
        \item The noncommutative torus, like the commutative torus, is not classically ergodic.
        \item The spectral triple of the Toeplitz algebra is classically ergodic. See~\cite[Example~(D)]{Zelditch1996} for a generalisation.  
    \end{enumerate}
\end{ex}
\begin{proof} $(1)$: Example~\ref{ex:L2SAspaces}.\ref{ex:commutativel2original} gave that $L_2S^*\A$ is isomorphic to $L_2(S^*M)$ in this setting, with $G_t$ given by the geodesic flow. The definition of classic ergodicity in Definition~\ref{def:ClassicErgodicity} is then equivalent with the standard definition of ergodic geodesic flow, see e.g.~\cite[Proposition~2.4.1]{Petersen1989}.\\
    $(2)$: Since $G_t$ acts on $L_2(S)\otimes HS_F$ by $G_t^M \otimes 1$, any element of the form $1 \otimes a$ is a fixed point of $G_t$.\\
    $(3)$: It follows from Example~\ref{ex:L2SAspaces} that for any $f \in C(\mathbb{S}^{d-1})$, the element $\syml(1\otimes f) \in L_2(S^*\mathbb{T}^d_\theta)$ is a fixed point of $G_t$.\\
    $(4)$: Since the only rotationally invariant functions in $L_2(\mathbb{S}^1)$ are the constant functions, the claim follows.
\end{proof}

We note that the well-studied examples of spectral triples in noncommutative geometry often possess a high degree of symmetry, and in geometric examples a high degree of symmetry can obstruct ergodicity.

\begin{ex}
    A noncommutative example where classical ergodicity has been demonstrated can be found in~\cite[Proposition~3.2]{MaMa2024}. This concerns operators on vector-bundle valued sections of a compact manifold. It is in a sense a more `twisted' version of the almost commutative manifolds in Example~\ref{ex:L2SAspaces2}.\ref{ex:AlmostComm2},  
\end{ex}

\begin{rem}
    In the context of Section~\ref{S:DOS}, it is unknown to the authors what the significance of the construction $S^*\A$ and the concept of ergodicity is. In that section, taking a discrete metric space $(X,d_X)$, the operator $D$ was taken to be $M_w^{-1}$, where $w:X\to\C$ is defined by
    \[
    w(x) := \frac{1}{|B(x_0, d_X(x,x_0))|}, \quad x\in X.
    \]
    It was shown in Proposition~\ref{P:PropCWeyl} that $M_w^{-1}$ satisfies a Weyl law if $X$ satisfies the condition
    \[
    \frac{|B(x_0,r_{k+1})|}{|B(x_0,r_k)|}\to 1, \quad k\to \infty.
    \]
    For the algebra $\A$, it makes sense to include some operators of the form $-\Delta + M_V$, where $\Delta$ is a (bounded) discrete analogue of the Laplace operator and $V:X\to \R$ is bounded. Then, however, $(\A, \Hc, D)$ has no chance to be classically ergodic: $M_w^{-1}$ commutes with the multiplication operators $M_V$, and hence $L_2(S^*\A)$ will contain many fixed points for the action $G_t$. For the weaker property of quantum ergodicity, we would need for the canonical eigenbasis $\{e_x\}_{x\in X}$ of $\ell_2(X)$, that there exists a density one subset $J \subseteq X$ such that
    \[
    \lim_{j \in J}\langle e_{j}, A e_{j} \rangle = \frac{\Tr_{\omega}(A \langle D\rangle^{-d})}{\Tr_{\omega}(\langle D\rangle^{-d})}, \quad A \in \A.
    \] 
\end{rem}

We now conclude this paper by giving some equivalent conditions for classical ergodicity.
First, we invoke von Neumann's mean ergodic theorem~\cite[Theorem~II.11]{ReedSimon1}.
\begin{prop}\label{P:ErgodicThms}
    For any $a \in L_2(S^*\A)$ there exists a fixed point of $G_t$ denoted by $a_{avg} \in L_2(S^*\A)$ such that, putting 
    \[
    a_T := \frac{1}{T}\int_{0}^T G_t(a) \, dt,
    \]
    we have
    \[
    \lim_{T\to \infty} \| a_T- a_{avg}\|_{L_2} \to 0.
    \]
    Furthermore,
    \[
    \langle 1, a_{avg}\rangle_{L_2} = \langle 1, a\rangle_{L_2},
    \]
    and the map $a\mapsto a_{avg}$ is $L_2$-continuous.
\end{prop}
\begin{proof}
    The existence of $a_{avg}$ and the $L_2$-convergence of $a_T$ to $a_{avg}$ follows from von Neumann's mean ergodic theorem, see e.g.~\cite[Theorem~II.11]{ReedSimon1}, or see~\cite[Corollary~VIII.7.3]{DunfordSchwartzI} for the continuous-time variant we use here. 
    
    Next, since $G_t$ is a unitary operator with $G_t(1) = 1$, we have 
    \[
    \langle 1, G_t(a)\rangle_{L_2} = \langle 1, a \rangle_{L_2}.
    \]
    Hence,
    \begin{align*}
        \big|\langle 1, a_{avg}\rangle_{L_2} - \langle 1, a\rangle_{L_2} \big| &\leq \big|\langle 1, a_{avg}\rangle_{L_2} - \langle1, a_T\rangle_{L_2} \big| + \smash{\underbrace{ \big|\langle1,a_T\rangle_{L_2} - \langle1,a\rangle_{L_2} \big|}_{=0}}\\
        &=|\langle 1, a_{avg} - a_T\rangle_{L_2}|\\
        &\leq \| a_T- a_{avg}\|_{L_2},
    \end{align*}
    which converges to $0$ as $T \to \infty$ due to the first part. 
    
    The element $a_{avg}$ being a fixed point of $G_t$ is a consequence of the estimate
    \begin{align*}
     \|a_T - G_t(a_T)\|_{L_2} &= \frac{1}{T} \bigg \| \int_0^tG_s(a)\, ds - \int_T^{T+t} G_s(a)\,ds\bigg\|_{L_2}\\
     &\leq \frac{2t}{T}\|a\|_{L_2}\xrightarrow{T \to \infty} 0.    
    \end{align*}
    Finally, for the continuity of $a \mapsto a_{avg}$, note that
    \[
    \|a_T\|_{L_2} \leq \frac{1}{T} \int_0^T \|G_t( a) \|_{L_2}\,dt =  \| a\|_{L_2},
    \]
    and taking the limit $T \to \infty$,
    \begin{align*}
        \|a_{avg} \|_{L_2}\leq \|a\|_{L_2}.
    \end{align*}
\end{proof}

\begin{rem}
    Since $S^*\A$ is represented as bounded operators on $L_2(S^*\A)$, we can consider the von Neumann algebra $\pi_\tau(S^*\A)''$ in $B(L_2(S^*\A))$, denoted as $L_\infty(S^*\A)$, to which $\tau$ extends as a faithful normal tracial state. 
    We can define the noncommutative $L_p$ spaces $L_p(S^*\A):= L_p(\tau)$ for $1 \leq p \leq \infty$ via standard constructions (we recover $L_2(S^*\A)$ for $p=2$). This is precisely how the spaces $L_p(\mathbb{T}^d_\theta)$ are constructed~\cite[Section~3.5]{LMSZVol2}. It is possible to show that $G_t: L_p(S^*\A) \to L_p(S^*\A)$ are isometries for all $1 \leq p \leq \infty$, and the averages in Proposition~\ref{P:ErgodicThms} exist and converge in every $L_p(S^*\A)$. 
\end{rem}

\begin{prop}\label{P:ClassicErgod}
    Given a unital regular spectral triple  $(\A, \Hc, D)$ satisfying Weyl's law, the following are equivalent:
    \begin{enumerate}
        \item the spectral triple is classically ergodic; \label{item:test}
        \item for all $a \in L_2(S^*\A)$,
    \[
     a_{avg} = \langle 1, a\rangle_{L_2} \cdot 1;
    \]
    \item writing 
    \[
    A_T := \frac{1}{T}\int_{0}^T \sigma_t(A) \, dt, \quad A \in \big \langle \bigcup_{t \in \R} \sigma_t(\mathcal{A}) \big \rangle,
    \]
    where $\big \langle \bigcup_{t \in \R} \sigma_t(\mathcal{A}) \big \rangle$ is the $*$-algebra generated by $\bigcup_{t \in \R} \sigma_t(\mathcal{A})$, we have
    for all $A \in \big \langle \bigcup_{t \in \R} \sigma_t(\mathcal{A}) \big \rangle$
    \[
    \lim_{T \to \infty} \omega \circ M \Big( \langle e_k, |A_T|^2 e_k \rangle \Big) = \bigg( \frac{ \Tr_{\omega} (A \langle D\rangle^{-d})}{\Tr_\omega(\langle D\rangle^{-d})}  \bigg)^2.
    \]
    \end{enumerate} 
\end{prop}
\begin{proof}
    $(1) \Leftrightarrow (2)$ is easily seen from the fact that $a_{avg}$ is a fixed point of $G_t$.\\
    Next, if $A \in \big \langle\bigcup_{t \in \R} \sigma_t(\mathcal{A})\big \rangle$, then by Theorem~\ref{T:Log-CesaroMean} it follows that 
    \begin{align*}
    \omega \circ M \Big( \langle e_k, |A_T|^2 e_k \rangle \Big) &= \frac{ \Tr_\omega(|A_T|^2\langle D\rangle^{-d})}{\Tr_\omega(\langle D\rangle^{-d})}\\
    &= \langle \syml(A_T), \syml(A_T)\rangle_{L_2}.
    \end{align*}
    Since $\syml(A_T) = \syml(A)_T$, Proposition~\ref{P:ErgodicThms} gives that
    \begin{equation}\label{eq:SunadaStep}
    \lim_{T \to \infty} \omega \circ M \Big( \langle e_k, |A_T|^2 e_k \rangle \Big) =  \langle \syml(A)_{avg}, \syml(A)_{avg}\rangle_{L_2}.
    \end{equation}
    $(2) \Rightarrow (3)$: This now follows from Equation~\eqref{eq:SunadaStep} and the identity $\langle 1, \syml(A) \rangle_{L_2} = \frac{\Tr_{\omega}(A\langle D\rangle^{-d})}{\Tr_\omega(\langle D\rangle^{-d})}$.\\
    $(3) \Rightarrow (2)$: For $a \in L_2(S^*\A)$, 
    Proposition~\ref{P:ErgodicThms} gives that $\langle a_{avg},1\rangle_{L_2} = \langle a, 1 \rangle_{L_2}$, and hence
    \begin{align*}
         \|a_{avg} - \langle 1,a\rangle \cdot 1\|_{L_2}^2
        &=  \langle a_{avg}, a_{avg}\rangle_{L_2} - \langle a_{avg}, 1\rangle_{L_2} \overline{\langle 1,a\rangle_{L_2}} - \langle 1, a_{avg}\rangle _{L_2} \langle 1,a\rangle _{L_2} + |\langle 1,a\rangle_{L_2}|^2 \\
        &= \langle a_{avg}, a_{avg}\rangle_{L_2} - |\langle 1,a\rangle_{L_2}|^2.
    \end{align*}
    Therefore, assumption (3) combined with Equation~\eqref{eq:SunadaStep} gives for all $A \in \big \langle \bigcup_{t \in \R} \sigma_t(\mathcal{A})\big\rangle$,
    \[
    \syml(A)_{avg} = \langle 1, \syml(A)\rangle_{L_2} \cdot 1.
    \]
    The image of $\big \langle \bigcup_{t \in \R} \sigma_t(\mathcal{A})\big\rangle$ under the map $\syml$ being dense in $L_2(S^*\A)$, and the map $a \to a_{avg}$ being $L_2$-continuous, we can conclude that
    \[
    a_{avg} = \langle 1, a\rangle_{L_2} \cdot 1
    \]
    for all $a \in L_2(S^*\A)$. 
\end{proof}

\printbibliography

\end{document}